\documentclass[12pt,reqno]{amsart}

\usepackage[vmargin=1in,hmargin=1.2in]{geometry}
\usepackage{graphicx, latexsym, graphics}
\usepackage{amsfonts, amssymb, amsmath, amsthm, mathtools, mathrsfs}
\usepackage{algorithmic, algorithm}

\newcommand{\NN}{\mathbb N}

\newcommand{\CC}{\mathbb C}
\newcommand{\RR}{\mathbb R}
\newcommand{\ZZ}{\mathbb Z}

\newcommand{\EE}{\mathcal E}
\newcommand{\DD}{\mathcal D}
\newcommand{\SSS}{\mathcal S}

\newcommand{\XX}{\mathcal X}

\newcommand{\supp}{\operatorname{supp}}

\theoremstyle{plain}
\newtheorem{theorem}{Theorem}[section]
\newtheorem{proposition}[theorem]{Proposition}
\newtheorem{lemma}[theorem]{Lemma}

\theoremstyle{remark}
\newtheorem{remark}[theorem]{Remark}

\theoremstyle{definition}

\allowdisplaybreaks

\numberwithin{equation}{section}

\begin{document}

\author[S. Pilipovi\' c]{Stevan Pilipovi\' c}
\address{Department of Mathematics and Informatics,
University of Novi Sad, Trg Dositeja Obradovi\'{c}a 4, 21000 Novi Sad, Serbia}
\email{stevan.pilipovic@dmi.uns.ac.rs}

\author[B. Prangoski]{Bojan Prangoski}
\thanks{The work of S. Pilipovi\'c and B. Prangoski was partially supported by the bilateral project ``Microlocal analysis and applications'' funded by the Macedonian Academy of Sciences and Arts and the Serbian Academy of Sciences and Arts.}
\address{Department of Mathematics, Faculty of Mechanical
Engineering-Skopje, Ss. Cyril and Methodius University in Skopje, Karposh 2 b.b., 1000 Skopje, Macedonia}
\email{bprangoski@yahoo.com}

\author[S. Tuti\' c]{Stefan Tuti\' c}
\thanks{The work of S.Tuti\' c was partially supported by the Science Fund of the Republic of Serbia, $\#$GRANT No 2727, \emph{Global and local analysis of operators and distributions} - GOALS. The author also acknowledges the support of the Ministry of Science, Technological Development and Innovation of the Republic of Serbia (Grants No. 451-03-33/2026-03/ 200125 $\&$ 451-03-34/2026-03/ 200125).}
\address{Department of Mathematics and Informatics,
University of Novi Sad, Trg Dositeja Obradovi\'{c}a 4, 21000 Novi Sad, Serbia}
\email{stefan.tutic@dmi.uns.ac.rs}

\title{On the pullback on spaces of distributions with prescribed microregularity with respect to a general Banach space}

\frenchspacing

\begin{abstract}
We consider the space $\DD'^E_L(U)$ of all distributions on the open set $U$ whose wave front set measured with respect to a Banach space $E$ lies in a closed conic subset $L$ of $U\times(\RR^n\backslash\{0\})$. The Banach space $E$ only satisfies mild technical assumptions. We show that the pullback by a diffeomorphism $f:O\rightarrow U$ is a topological isomorphism $f^*:\DD'^E_L(U)\rightarrow\DD'^E_{f^*L}(O)$. In the case when $E$ is the $L^p$-Sobolev space $W^{r,p}(\RR^n)$ of order $r\in\RR$, we study the pullback by a smooth map $f:O\rightarrow U$ of constant rank.
\end{abstract}


\maketitle

\section{Introduction}

Given a smooth map $f:O\rightarrow U$ between the open sets $O$ and $U$ in $\RR^m$ and $\RR^n$ respectively, the pullback by $f$ is the continuous linear operator $f^*:\mathcal{C}^{\infty}(U)\rightarrow \mathcal{C}^{\infty}(O)$, $f^*u:=u\circ f$. In \cite[Theorem 2.5.11']{hormander} and \cite[Theorem 8.2.4, p. 263]{hor}, H\"ormander showed that $f^*$ can be extended to distributions $u\in\DD'(U)$ as long as their $\mathcal{C}^{\infty}$-wave front set $WF(u)$ does not intersect the so-called set of normals $\mathcal{N}_f$ of $f$ (see \eqref{nor-toa-mapinskl} below). He considered the space $\DD'_L(U):=\{u\in\DD'(U)\,|\, WF(u)\subseteq L\}$ with $L$ a closed conic subset of $U\times(\RR^n\backslash\{0\})$ (the cotangent bundle of $U$ without the zero section) and showed that $f^*$ uniquely extends to a well-defined and sequentially continuous mapping $f^*:\DD'_L(U)\rightarrow \DD'_{f^*L}(O)$ when $\mathcal{N}_f\cap L=\emptyset$; $\DD'_L(U)$ is equipped with an appropriate locally convex topology. Later, in \cite{D1,DB}, the authors introduced a stronger topology on $\DD'_L(U)$ and, among other things, they showed that the H\"ormander extension $f^*:\DD'_L(U)\rightarrow \DD'_{f^*L}(O)$ is in fact continuous with respect to this new topology. Recently, in \cite[Theorem 3.21]{PP}, a refinement of the H\"ormander result was shown: assuming $\mathcal{N}_f\cap L=\emptyset$, the pullback extends to a well-defined and continuous map $f^*:\DD'^{r_2}_L(U)\rightarrow \DD'^{r_1}_{f^*L}(O)$ when $r_2-r_1>n/2$ and $r_2>n/2$ where $\DD'^r_L(U)$ is the space of all distributions $u\in\DD'(U)$ whose $L^2$-Sobolev wave front set $WF^r(u)$ of order $r\in\RR$ is contained in $L$ (this is a refinement of the H\"ormander result since $\DD'_L(U)=\bigcap_{r>0}\DD'^r_L(U)$). Furthermore, \cite[Theorem 3.21]{PP} also considers the case when $f$ has a constant rank $k\in\ZZ_+$ and in this case one can be more precise with the bounds on $r_2$ and $r_1$, namely $r_2-r_1\geq(n-k)/2$ and $r_2> (n-k)/2$; in \cite[Appendix A]{PP} it is shown that these conditions on $r_2$ and $r_1$ are essentially optimal. A natural problem is to extend this result to the corresponding spaces defined via the $L^p$-Sobolev wave front sets. In this article, we address this problem. We first define and study the space $\DD'^E_L(U)$ consisting of the distribution whose wave front set measured with respect to a Banach space $E$ is contained in $L$; here $E$ satisfies certain mild technical assumptions. When $E$ is the $L^2$-Sobolev space $H^r(\RR^n)$, of order $r\in\RR$, one recovers $\DD'^r_L(U)$ while by taking $E$ to be the $L^p$-Sobolev space $W^{r,p}(\RR^n)$, $1<p<\infty$, of order $r\in\RR$ we obtain the space to which we wish to extend the pullback by $f$. Of course, one can consider other special cases of $E$; e.g. $E$ can be the H\"older spaces of non integer order or the Besov spaces in which case $\DD'^E_L(U)$ contains the distributions whose Besov wave front set (studied in \cite{dap-r-scl}) is contained in $L$. The main result of the article is Theorem \ref{the-for-pulbknosklskhtrk}. Here we specialise $E=W^{r,p}(\RR^n)$, $1<p<\infty$, and denoting $\DD'^E_L(U)$ by $\DD'^{r,p}_L(U)$, we show that the pullback by a map of constant rank $k\in\ZZ_+$ extends to a well-defined and continuous map $f^*:\DD'^{r_2,p}_L(U)\rightarrow \DD'^{r_1,p}_{f^*L}(O)$ when $\mathcal{N}_f\cap L=\emptyset$, $r_2-r_1>(n-k)/p$ and $r_2>(n-k)/p$ which is in-line with the conditions on $r_2$ and $r_1$ in the $L^2$-Sobolev case in \cite[Theorem 3.21]{PP}. We point out that a significant difficulty which appears in the proof of Theorem \ref{the-for-pulbknosklskhtrk} is the fact that $\mathcal{F}L^p(\RR^n)$ is not solid when $p\neq 2$ (cf. \cite[Theorem (b)]{fefferman}) unlike the case $p=2$ when $\mathcal{F}L^2(\RR^n)=L^2(\RR^n)$. The most useful applications of Theorem \ref{the-for-pulbknosklskhtrk} are when $f$ is either an imbedding of or a projection on a lower dimensional hyperplane in which case $f^*$ is restricting distributions to a lower dimensional hyperplane or extending distributions to an open set in higher dimension respectively (see the comment at the very end of the article). Finally, it is likely that Theorem \ref{the-for-pulbknosklskhtrk} may be employed in extending the (pointwise) product to distributions as in \cite[Theorem 8.2.10, p. 267]{hor} and in \cite{Tutic}.

\section{Preliminaries}

We fix the constants in the Fourier transform as $\mathcal{F}f(\xi):=\int_{\RR^n}e^{-ix\xi}f(x)dx$, $f\in L^1(\RR^n)$. We write $\check{f}(x)=f(-x)$ for the reflection operator. Throughout the article, $B(x,r)$ stands for the open ball with centre at $x$ and radius $r>0$. For any $A\subseteq \RR^n$, we denote by $\mathbf{1}_A$ the function which equals $1$ on $A$ and $0$ on $\RR^n\backslash A$. As standard, $\langle x\rangle:=(1+|x|^2)^{1/2}$, $x\in\RR^n$. We denote $\RR_+:=\{\lambda\in\RR\,|\, \lambda>0\}$ and, for any $B\subseteq \RR^n$, $\RR_+B$ stands for the cone $\{\lambda x\in\RR^n\,|\, \lambda\in\RR_+,\, x\in B\}$. The cotangent bundle without the zero section of the open set $U\subseteq \RR^n$ is canonically identified with $U\times (\RR^n\backslash\{0\})$. A set $L\subseteq U\times (\RR^n\backslash\{0\})$ is said to be conic if $(x,\xi)\in L$ implies $(x,\lambda\xi)\in L$, $\lambda>0$ (i.e. it is a cone with respect to the second variable for each fixed $x\in U$). For such $L$, we denote $L^c:=(U\times (\RR^n\backslash\{0\}))\backslash L$. Given two open sets $O$ and $U$ in $\RR^m$ and $\RR^n$ respectively and a $\mathcal{C}^{\infty}$ map $f:O\rightarrow U$, we denote by $f'(x)$ the derivative of $f$ at $x\in O$. Following H\"ormander \cite{hor}, we denote
\begin{equation}\label{nor-toa-mapinskl}
\mathcal{N}_f:=\{(f(x),\eta)\in U\times \RR^n\,|\, x\in O,\,{}^tf'(x)\eta=0\}
\end{equation}
and, given a closed conic subset $L$ of $U\times(\RR^n\backslash\{0\})$ which satisfies $L\cap \mathcal{N}_f=\emptyset$,
\begin{equation*}
f^*L:=\{(x,{}^tf'(x)\eta)\in O\times \RR^m\,|\, (f(x),\eta)\in L\}
\end{equation*}
is a closed conic subset of $O\times(\RR^m\backslash\{0\})$. When $m=n$, we set $|f'|(x):=|\det f'(x)|$, $x\in O$. If $f$ is a diffeomorphism, the pullback $f^*:\mathcal{C}^{\infty}(U)\rightarrow\mathcal{C}^{\infty}(O)$, $f^*\varphi:=\varphi\circ f$, uniquely extends to the well-defined and continuous map
\begin{equation*}
f^*:\DD'(U)\rightarrow \DD'(O),\quad \langle f^*u,\varphi\rangle:=\langle u,|f^{-1\, '}|\varphi\circ f^{-1}\rangle,\,\, u\in\DD'(U),\, \varphi\in\DD(O),
\end{equation*}
and the latter restricts to a continuous map $f^*:\EE'(U)\rightarrow\EE'(O)$. \\
\indent For $0\leq \rho\leq1$ and $r\in\RR$, we denote by $S^r_{\rho,0}(\RR^{2n})$ the H\"ormander space of symbols on $\RR^{2n}$ \cite{hormander,hor2}: it consists of all $a\in\mathcal{C}^{\infty}(\RR^{2n})$ satisfying $\sup_{x,\xi\in\RR^n}\langle\xi\rangle^{-r+\rho|\alpha|}|\partial^{\alpha}_{\xi}\partial^{\beta}_xa(x,\xi)|<\infty$, $\alpha,\beta\in\NN^n$. As standard, we denote $S^{-\infty}(\RR^{2n}):=\bigcap_{r\in\RR}S^r_{\rho,0}(\RR^{2n})$; clearly, the definition is independent of $\rho\in[0,1]$. The pseudo-differential operator with symbol $a\in S^r_{\rho,0}(\RR^{2n})$ is defined as
\begin{equation*}
a(x,D)\varphi(x):=\frac{1}{(2\pi)^n}\int_{\RR^n}e^{i x\xi}a(x,\xi)\mathcal{F}\varphi(\xi)d\xi,\quad \varphi\in\SSS(\RR^n).
\end{equation*}
It is a continuous operator on $\SSS(\RR^n)$, it extends to a continuous operator on $\SSS'(\RR^n)$ and, when $r=0$, $a(x,D)$ is also a continuous operator on $L^2(\RR^n)$; we refer to \cite{hor2} for its further properties. We point out that if $a\in S^0_{1,0}(\RR^{2n})$, then $a(x,D)$ is also a continuous operator on $L^p(\RR^n)$, $1<p<\infty$. When $a\in S^0_{0,0}(\RR^{2n})$ does not depend on $x$, i.e. $a(x,\xi)=a(\xi)$, the $\Psi$DO is a Fourier multiplier and we write $a(D)$ instead of $a(x,D)$; notice that $a(D)\varphi=\mathcal{F}^{-1}(a)*\varphi$, $\varphi\in\SSS(\RR^n)$. We denote by $W^{r,p}(\RR^n)$, $1<p<\infty$, the $L^p$-Sobolev space of order $r\in\RR$, i.e. $W^{r,p}(\RR^n):=\{u\in\SSS'(\RR^n)\,|\, \mathcal{F}^{-1}(\langle\cdot\rangle^r \mathcal{F}u)\in L^p(\RR^n)\}$, and we point out that the $\Psi$DOs with symbols in $S^0_{1,0}(\RR^{2n})$ are continuous operators on $W^{r,p}(\RR^n)$.\\
\indent For any measurable $w:\RR^n\rightarrow(0,\infty)$ which satisfies $C^{-1}\langle x\rangle^{-\tau}\leq w(x)\leq C\langle x\rangle^{\tau}$, a.a. $x$, for some $C,\tau> 0$, we denote by $L^p_w(\RR^n)$, $1\leq p\leq \infty$, the weighted $L^p$-space $L^p_w(\RR^n):=\{f\in L^1_{\operatorname{loc}}(\RR^n)\,|\, wf\in L^p(\RR^n)\}$ with norm $\|f\|_{L^p_w(\RR^n)}:=\|wf\|_{L^p(\RR^n)}$.\\
\indent A function $\psi:\RR^n\rightarrow \CC$ is said to be positively homogeneous of order $0$ away from the origin if there is $R>0$ such that $\psi(\lambda x)=\psi(x)$, $|x|>R$, $\lambda>1$. We denote by $\XX(\RR^n)$ the space of smooth positively homogeneous functions of order $0$ away from the origin. If $\psi\in\XX(\RR^m)$, then $\langle\cdot\rangle^{|\alpha|} \partial^{\alpha}\psi\in L^{\infty}(\RR^n)$, $\alpha\in\NN^n$. Hence, $(x,\xi)\mapsto \psi(\xi)$ belongs to $S^0_{1,0}(\RR^{2n})$ and consequently $\psi(D)$ is continuous on $L^p(\RR^n)$, $1<p<\infty$.\\
\indent When $X$ and $Y$ are locally convex spaces (in short: l.c.s.), we denote by $\mathcal{L}(X,Y)$ the space of continuous linear operators from $X$ into $Y$. We employ the notation $\mathcal{L}_b(X,Y)$ for the space $\mathcal{L}(X,Y)$ equipped with the topology of uniform convergence on all bounded sets while $\mathcal{L}_{\sigma}(X,Y)$ will denote that $\mathcal{L}(X,Y)$ is equipped with the topology of simple (pointwise) convergence. When $X=Y$, we simply write $\mathcal{L}(X)$, $\mathcal{L}_b(X)$ and $\mathcal{L}_{\sigma}(X)$ instead.

\subsection{Wave front sets with respect to a class of Banach spaces of distributions}

We recall the following class of Banach spaces of distributions from \cite{DPPV-TMIB}; throughout the section, $T_x$, $x\in\RR^n$, stands for the operator of translation $T_xf:=f(\cdot-x)$ and $M_{\xi}$, $\xi\in\RR^n$, stands for the operator of modulation $M_{\xi}f:=e^{i\xi\,\cdot}f$. The Banach space $E$ with norm $\|\cdot\|_E$ is said to be a \textit{translation-modulation invariant Banach space of distributions} (abbreviated as (TMIB)-space) if it satisfies the following conditions:
\begin{itemize}
\item[$(i)$] $\SSS(\RR^n)\subseteq E\subseteq \SSS'(\RR^n)$ with continuous and dense inclusions;
\item[$(ii)$] $T_x\in\mathcal{L}(E)$, $x\in\RR^n$, and $M_{\xi}\in\mathcal{L}(E)$, $\xi\in\RR^n$;
\item[$(iii)$] there are $\tau,C>0$ such that
\begin{equation*}
\omega_E(x):=\|T_x\|_{\mathcal{L}_b(E)}\leq C(1+|x|)^{\tau}\quad \mbox{and}\quad \nu_E(\xi):=\|M_{-\xi}\|_{\mathcal{L}_b(E)}\leq C(1+|\xi|)^{\tau}.
\end{equation*}
\end{itemize}
We recall here the properties of these spaces that we employ throughout the article and we defer to \cite{DPPV-TMIB} for the complete account. The functions $\omega_E$ and $\nu_E$ defined in $(iii)$ are submultiplicative and Borel measurable. The convolution mapping $*:\SSS(\RR^n)\times\SSS(\RR^n)\rightarrow \SSS(\RR^n)$ and the multiplication mapping $\cdot:\SSS(\RR^n)\times\SSS(\RR^n)\rightarrow\SSS(\RR^n)$ uniquely extend to the following continuous bilinear mappings:
\begin{align}
&*:L^1_{\omega_E}(\RR^n)\times E\rightarrow E,\quad \|f*e\|_E\leq \|f\|_{L^1_{\omega_E}(\RR^n)}\|e\|_E,\,\, f\in L^1_{\omega_E}(\RR^n),\, e\in E,\label{con-module-ine}\\
&\cdot:\mathcal{F}L^1_{\nu_E}(\RR^n)\times E\rightarrow E,\quad \|g\cdot e\|_E\leq \|g\|_{\mathcal{F}L^1_{\nu_E}(\RR^n)}\|e\|_E,\,\, g\in\mathcal{F}L^1_{\nu_E}(\RR^n),\, e\in E.\label{mult-module-ine}
\end{align}
Consequently, $E$ becomes a Banach convolution module over the Beurling convolution algebra $L^1_{\omega_E}(\RR^n)$ and a Banach multiplication module over the Fourier-Beurling multiplication algebra $\mathcal{F}L^1_{\nu_E}(\RR^n)$ (the multiplication in $\mathcal{F}L^1_{\nu_E}(\RR^n)$ is defined by $g_1\cdot g_2:=\mathcal{F}(\mathcal{F}^{-1}g_1*\mathcal{F}^{-1}g_2)$; when $\nu_E\geq 1$, $\mathcal{F}L^1_{\nu_E}(\RR^n)$ consists of continuous functions and this multiplication coincides with the ordinary pointwise multiplication of continuous functions). Furthermore, $\mathcal{F}E$ (with the induced norm $\|v\|_{\mathcal{F}E}:=\|\mathcal{F}^{-1}v\|_E$, $v\in\mathcal{F}E$) is again a (TMIB)-space with $\omega_{\mathcal{F}E}=\check{\nu}_E$ and $\nu_{\mathcal{F}E}=\omega_E$. Typical examples of (TMIB)-spaces are $L^p_{\langle\cdot\rangle^s}(\RR^n)$, $1\leq p<\infty$, $s\in\RR$, and the space of continuous functions which vanish at infinity $\mathcal{C}_0(\RR^n)$ (below we will give more examples).\\
\indent The Banach space $E$ is said to be a \textit{dual translation-modulation invariant Banach space of distributions} (abbreviated as (DTMIB)-space) if it is the strong dual of a (TMIB)-space $E_0$. In this case $E$ again satisfies the continuous inclusions $\SSS(\RR^n)\subseteq E\subseteq \SSS'(\RR^n)$ but $\SSS(\RR^n)$ may fail to be dense in $E$ (e.g. $E=L^{\infty}(\RR^n)$). Furthermore, $E$ always satisfies $(ii)$ and $(iii)$ and $\omega_E=\check{\omega}_{E_0}$ and $\nu_E=\nu_{E_0}$. One defines the convolution with elements in $L^1_{\omega_E}(\RR^n)$ and the multiplication with elements in $\mathcal{F}L^1_{\nu_E}(\RR^n)$ via duality: $\langle f*e,e_o\rangle:=\langle e,\check{f}*e_0\rangle$, $e\in E$, $f\in L^1_{\omega_E}(\RR^n)$, $e_0\in E_0$, and $\langle g\cdot e,e_0\rangle:=\langle e,g\cdot e_0\rangle$, $e\in E$, $g\in \mathcal{F}L^1_{\nu_E}(\RR^n)$, $e_0\in E_0$. With this, $E$ also satisfies \eqref{con-module-ine} and \eqref{mult-module-ine}, i.e. it becomes a convolution module over $L^1_{\omega_E}(\RR^n)$ and a multiplication module over $\mathcal{F}L^1_{\nu_E}(\RR^n)$. Typical examples of (DTMIB)-spaces are $L^p_{\langle\cdot\rangle^s}(\RR^n)$, $1< p\leq\infty$, $s\in\RR$, and the space of bounded Radon measures $\mathcal{M}^1(\RR^n)$.\\
\indent Let $E$ be a (TMIB) or a (DTMIB)-space. We are going to consider wave front sets which collect the directions where a given distribution does not behave locally as an elements of $E$. For this purpose, we impose the following additional condition introduced in \cite[Section 3]{DPP} (we refer to \cite[Section 3]{DPP} for the reasons why such condition is needed for the wave front set to behave in a natural way):
\begin{itemize}
\item[$(iv)$] The elements of $\XX(\RR^n)$ are Fourier multipliers for $E$, i.e for each $\chi\in\XX(\RR^n)$, the map $\mathcal{F}E\rightarrow \mathcal{F}E$, $f\mapsto \chi f$, is well-defined and continuous (equivalently, the map $E\rightarrow E$, $e\mapsto \chi(D) e$, is well-defined and continuous).
\end{itemize}
Given $u\in\DD'(U)$, $U$ open set in $\RR^n$, the wave front set of $u$ with respect to $E$, in notation $WF^E(u)$, is defined as follows; see \cite[Section 3]{DPP}. For each $\varphi\in\DD(U)$, we first define the closed cone $\Sigma^E(\varphi u)$ in $\RR^n\backslash\{0\}$ be declaring that $\xi\in\RR^n\backslash\{0\}$ does not belong to $\Sigma^E(\varphi u)$ if there is an open cone $V$ in $\RR^n\backslash\{0\}$ containing $\xi$ and a nonnegative $\psi\in\XX(\RR^n)$ satisfying $\psi=\operatorname{const.}>0$ on $\overline{V}\backslash B(0,R)$ for some $R>0$ such that $\psi\mathcal{F}(\varphi u)\in\mathcal{F}E$. For each $x\in U$, set $\Sigma^E_x(u):=\bigcap_{\varphi\in\DD(U),\, \varphi(x)\neq0}\Sigma^E(\varphi u)$ and define
\begin{equation*}
WF^E(u):=\{(x,\xi)\in U\times(\RR^n\backslash\{0\})\,|\, \xi\in \Sigma^E_x(u)\}.
\end{equation*}
Then $WF^E(u)$ is a closed conic subset of $U\times(\RR^n\backslash\{0\})$ and we refer to \cite[Section 3]{DPP} that $WF^E(u)$ indeed behaves in a natural way.\footnote{Aside from the condition $(iv)$ which is the same as in \cite[Section 3]{DPP}, the assumption that $E$ is a (TMIB) or a (DTMIB)-space readily imply that $E$ satisfies the conditions \cite[$(i)$-$(iii)$ in Section 3]{DPP}. We will need the slightly stronger assumption that $E$ is (TMIB) or a (DTMIB)-space for some of the results in the article.} Of course, $L^p_{\langle\cdot\rangle^s}(\RR^n)$, $1<p<\infty$, $s\in\RR$, satisfy $(iv)$, however $L^1_{\langle\cdot\rangle^s}(\RR^n)$, $L^{\infty}_{\langle\cdot\rangle^s}(\RR^n)$, $s\in\RR$, and $\mathcal{C}_0(\RR^n)$ do not; see \cite[Theorem 2.10]{DPP}. Also, the $L^p$-Sobolev spaces $W^{r,p}(\RR^n)$, $1<p<\infty$, $r\in\RR$, are (TMIB)-space that satisfy $(iv)$; when $E=W^{r,p}(\RR^n)$, $WF^E(u)$ is the ordinary $L^p$-wave front set of order $r\in\RR$ of $u$, usually denoted by $WF^{r,p}(u)$ (when $p=2$, this is further abbreviated as $WF^r(u)$). In addition the Besov spaces $B^s_{p,q}(\RR^n)$, $s\in\RR$, $p,q\in[1,\infty)$, are (TMIB)-spaces which satisfy $(iv)$ while $B^s_{\infty,\infty}(\RR^n)$, $s\in\RR$, is a (DTMIB)-space which satisfies $(iv)$ (see \cite[Theorem 2.17, p. 257, and Corollary 5.2, p. 608]{sawano} and \cite[Theorem, p. 140]{triebel}); consequently, the H\"older spaces of non-integer order $s>0$ are (DTMIB)-space which satisfy $(iv)$ since they coincide with $B^s_{\infty,\infty}(\RR^n)$.\\
\indent We end the section with the following technical result which we will tacitly employ throughout the rest of the article.

\begin{lemma}\label{lem-for-tmbsjk}
The (TMIB)-space $E$ satisfies $(iv)$ if and only if $E'$ does so.
\end{lemma}

\begin{proof}
It is straightforward to show that $E'$ satisfies $(iv)$ when $E$ does so. Assume that $E'$ satisfies $(iv)$ and let $\chi\in\XX(\RR^n)$. It suffices to show $\chi(D)E\subseteq E$ since then the continuity will follow from the closed graph theorem. Let $e\in E$ and pick a sequence $(\varphi_j)_{j\in\ZZ_+}$ in $\SSS(\RR^n)$ which converges to $e$ in $E$. Then
\begin{equation*}
|\langle e',\chi(D)\varphi_j-\chi(D)\varphi_k\rangle|\leq \|\check{\chi}(D)e'\|_{E'}\|\varphi_j-\varphi_k\|_E\leq C\|e'\|_{E'}\|\varphi_j-\varphi_k\|_E,
\quad e'\in E',
\end{equation*}
which implies $\|\chi(D)\varphi_j-\chi(D)\varphi_k\|_E\leq C\|\varphi_j-\varphi_k\|_E$. Whence, $(\chi(D)\varphi_j)_{j\in\ZZ_+}$ is a Cauchy sequence in $E$. Since it converges to $\chi(D)e$ in $\SSS'(\RR^n)$, we deduce $\chi(D)e\in E$ and the proof is complete.
\end{proof}

\section{The space of distributions with $E$-wave front in a fixed conic set}\label{Sec1}

Let $U$ be an open set in $\RR^n$ and $L$ a closed conic subset of $U\times(\RR^n\backslash\{0\})$. We define the space
\begin{equation*}
\DD'^E_L(U):=\{u\in\DD'(U)\,|\, WF^E(u)\subseteq L\}.
\end{equation*}
In what follows, we will frequently employ the following subspace of $\XX(\RR^n)$:
\begin{equation*}
\XX_0(\RR^n):=\{\chi\in\XX(\RR^n)\,|\, 0\notin\supp\chi\}.
\end{equation*}
Reasoning as in the proof of \cite[Lemma 8.2.1, p. 262]{hor} and applying \cite[Proposition 3.6]{DPP} one verifies the following fact.

\begin{lemma}\label{lem-for-semwelds}
Let $U$ be an open set in $\RR^n$ and $L$ a closed conic subset of $L\subseteq U\times (\RR^n\backslash\{0\})$. Then $u\in\DD'(U)$ belongs to $\DD'^E_L(U)$ if and only if for every $\varphi\in\DD(U)$ and every $\psi\in\XX_0(\RR^n)$ satisfying $(\supp\varphi\times\supp\psi)\cap L=\emptyset$ it holds that $\psi\mathcal{F}(\varphi u)\in \mathcal{F}E$.
\end{lemma}

We equip $\DD'^E_L(U)$ with the locally convex topology induced by all continuous seminorms on $\DD'(U)$ together with all seminorms
\begin{equation}
\label{seminorm}
\mathfrak{p}_{E;\varphi,\psi}(u):=\|\psi\mathcal{F}(\varphi u)\|_{\mathcal F E},
\end{equation}
where $\varphi\in \DD(U)$ and $\psi\in\XX_0(\RR^m)$ satisfy $(\supp\varphi\times\supp\psi)\cap L=\emptyset$. Clearly, $\mathcal{C}^{\infty}(U)\subseteq \DD'^E_{L}(U)\subseteq \DD'(U)$ continuously; in fact $\{v\in\DD'(U)\,|\, \varphi v\in E,\, \forall\varphi\in \DD(U)\}\subseteq \DD'^E_L(U)$.

\begin{remark}
Let $\varphi\in\DD(U)$ and $\psi\in\XX(\RR^n)$ be such that $(\supp\varphi\times\supp\psi)\cap L=\emptyset$. Pick $\chi\in\mathcal{C}^{\infty}(\RR^n)$ such that $0\leq \chi\leq1$, $\chi=0$ on $B(0,1)$ and $\chi=1$ on $\RR^n\backslash B(0,2)$. For $u\in\DD'(U)$, write $\psi\mathcal{F}(\varphi u)=(1-\chi)\psi\mathcal{F}(\varphi u)+\chi\psi\mathcal{F}(\varphi u)$ and notice that $(1-\chi)\psi\mathcal{F}(\varphi u)\in\DD(\RR^n)\subseteq \mathcal{F}E$. Hence $\psi\mathcal{F}(\varphi u)\in\mathcal{F}E$ if and only if $\chi\psi\mathcal{F}(\varphi u)\in\mathcal{F}E$ which implies that Lemma \ref{lem-for-semwelds} is true even with $\XX(\RR^n)$ in place of $\XX_0(\RR^n)$. Furthermore, since $\DD'(U)\rightarrow [0,\infty)$, $v\mapsto \|(1-\chi)\psi\mathcal{F}(\varphi v)\|_{\mathcal{F}E}$, is a continuous seminorm on $\DD'(U)$, we deduce that $\mathfrak{p}_{E;\varphi,\psi}(u):=\|\psi\mathcal{F}(\varphi u)\|_{\mathcal F E}$ is also a continuous seminorm on $\DD'^E_L(U)$.
\end{remark}

\begin{remark}
The reason why we used $\XX_0(\RR^n)$ in the definition of the topology of $\DD'^E_L(U)$ is the following useful property: For $\chi\in\XX_0(\RR^n)$, the cone $\{0\}\cup\RR_+\supp\chi$ is closed in $\RR^n$; we will always tacitly employ this fact throughout the rest of the article. To verify it, first notice that $\{0\}\cup\RR_+\supp\chi=\{0\}\cup\RR_+(\supp\chi\cap \overline{B(0,R)})$, for $R>0$ large enough. As $\supp\chi\cap \overline{B(0,R)}$ is compact and does not contain $0$, it is straightforward to show that $\{0\}\cup\RR_+(\supp\chi\cap \overline{B(0,R)})$ is closed.\footnote{If $0\in\supp\chi$, then $\{0\}\cup\RR_+(\supp\chi\cap \overline{B(0,R)})$ may fail to be closed in $\RR^n$.}
\end{remark}

\begin{remark}
Notice that $\DD'^E_{U\times(\RR^n\backslash\{0\})}(U)=\DD'(U)$ as l.c.s. On the other hand, if $E=H^r(\RR^n)$, $r\in\RR$, then $\DD'^E_L(U)$ is exactly the space $\DD'^r_L(U)$ considered in \cite{PP}.
\end{remark}

The following technical result will prove useful throughout the rest of the article.

\begin{lemma}\label{lem-for-est-oftheremininestwavefrps}
Let $\chi\in\SSS(\RR^n)$ and let $\psi_1,\psi_2\in\DD_{L^{\infty}}(\RR^n)$ satisfy the following: There are closed cones $V_1$ and $V_2$ in $\RR^n$ and $R>0$ such that $V_1\cap V_2=\{0\}$ and $\supp\psi_j\backslash B(0,R)\subseteq V_j$, $j=1,2$. Then
\begin{equation}\label{map-con-onskl}
\lim_{\substack{\longrightarrow \\ s\rightarrow\infty}}L^1_{\langle \cdot\rangle^{-s}}(\RR^n)\rightarrow \SSS(\RR^n),\quad v\mapsto \psi_2(\chi*(\psi_1v)),
\end{equation}
is a well-defined and continuous map.
\end{lemma}

\begin{proof}
When $V_1=\{0\}$ or $V_2=\{0\}$ the claim is trivial. Assume that both $V_1\backslash\{0\}\neq \emptyset$ and $V_2\backslash\{0\}\neq\emptyset$. There is $0<\varepsilon<1/2$ such that $|\xi-\xi'|\geq 2\varepsilon$, $\xi'\in V_1\cap\mathbb{S}^{n-1}$, $\xi\in V_2\cap\mathbb{S}^{n-1}$. For general $\xi'\in V_1\backslash\{0\}$ and $\xi\in V_2\backslash\{0\}$, it holds that
\begin{equation*}
\left|\frac{\xi'}{|\xi'|}-\frac{\xi}{|\xi'|}\right|\geq \left|\frac{\xi'}{|\xi'|}-\frac{\xi}{|\xi|}\right|-\left|\frac{\xi}{|\xi|}-\frac{\xi}{|\xi'|}\right|\geq 2\varepsilon-\frac{\left||\xi'|-|\xi|\right|}{|\xi'|}\geq 2\varepsilon-\frac{|\xi'-\xi|}{|\xi'|}.
\end{equation*}
Consequently, $|\xi'-\xi|\geq \varepsilon|\xi'|$; this trivially holds even when $\xi=0$ or $\xi'=0$.\\
\indent Clearly, the right-hand side of \eqref{map-con-onskl} is a well-defined smooth function on $\RR^n$ which we denote by $I_v$. Let $s>0$ and let $B$ be a bounded subset of $L^1_{\langle\cdot\rangle^{-s}}(\RR^n)$. For $v\in B$, $l>0$ and $\alpha\in\NN^n$, we infer
\begin{equation}\label{ine-for-operashconts}
\langle\xi\rangle^l|\partial^{\alpha}I_v(\xi)|\leq C\sum_{\beta\leq \alpha} \int_{\RR^n}|\partial^{\beta}\chi(\eta)|\langle\eta\rangle^l\langle\xi-\eta\rangle^l|\partial^{\alpha-\beta}\psi_2(\xi)| |\psi_1(\xi-\eta)||v(\xi-\eta)|d\eta.
\end{equation}
The right-hand side of \eqref{ine-for-operashconts} is uniformly bounded for $|\xi|\leq 2R$ and equals $0$ when $\xi\in \RR^n\backslash (V_2\cup B(0,2R))$. Let $\xi\in V_2\backslash B(0,2R)$. The domain of the integration can be reduced to $\{\eta\in\RR^n\,|\,\xi-\eta\in V_1\}\cup\{\eta\in\RR^n\,|\, |\xi-\eta|\leq R\}$. On the first set, the above inequality (applied with $\xi'=\xi-\eta$) implies $|\eta|\geq \varepsilon|\xi-\eta|$, while on the second we trivially have $|\xi-\eta|\leq R\leq |\eta|$. Consequently, the right-hand side of \eqref{ine-for-operashconts} is uniformly bounded when $\xi\in V_2\backslash B(0,R)$. We deduce that the map $L^1_{\langle \cdot\rangle^{-s}}(\RR^n)\rightarrow \SSS(\RR^n)$, $v\mapsto \psi_2(\chi*(\psi_1v))$, is well-defined and maps bounded sets into bounded sets. Consequently, it is continuous. Since $s>0$ was arbitrary, \eqref{map-con-onskl} is well-defined and continuous.
\end{proof}

The next proposition better describes the topology of $\DD'^E_L(U)$. Its main ingredients and its proof are similar to \cite[Proposition 3.7]{PP}. To state it, we need the following objects. Assume that $L^c\neq \emptyset$. Pick a countable dense subset $\{(x^{(j)},\xi^{(j)})\}_{j\in\ZZ_+}$ of $L^c$ and denote $\omega^{(j)}:=\xi^{(j)}/|\xi^{(j)}|\in\mathbb{S}^{n-1}$, $j\in\ZZ_+$. Set\footnote{Here we employ $\operatorname{dist}(x,\emptyset)=\infty$, for any element $x$.}
$$
s_j:=\min\{1,\operatorname{dist}(x^{(j)},\partial U),\operatorname{dist}((x^{(j)},\omega^{(j)}),L)\}>0.
$$
Clearly, $(\overline{B(x^{(j)},s_j/2)}\times \overline{B(\omega^{(j)},s_j/2)})\cap L=\emptyset$ and $B(x^{(j)},s_j)\subseteq U$, $j\in\ZZ_+$. For each $j,k\in\ZZ_+$, we set:
\begin{gather*}
O_{j,k}:=B(x^{(j)},s_j/(5k)),\quad O'_{j,k}:=B(x^{(j)},s_j/(4k)),\quad O''_{j,k}:=B(x^{(j)},s_j/(3k));\\
V_{j,k}:=\RR_+ B(\omega^{(j)},s_j/(5k)),\quad V'_{j,k}:=\RR_+ B(\omega^{(j)},s_j/(4k)),\quad V''_{j,k}:=\RR_+ B(\omega^{(j)},s_j/(3k)).
\end{gather*}
Notice that
$$
\overline{O_{j,k}}\times \overline{V_{j,k}}\subseteq O'_{j,k}\times (V'_{j,k}\cup\{0\}),\quad \overline{O'_{j,k}}\times \overline{V'_{j,k}}\subseteq O''_{j,k}\times (V''_{j,k}\cup\{0\}),\quad (\overline{O''_{j,k}}\times \overline{V''_{j,k}})\cap L=\emptyset.
$$
For each $j,k\in\ZZ_+$, pick $\phi_{j,k}\in\DD(\RR^n)$ and $\widetilde{\phi}_{j,k}\in\XX_0(\RR^n)$ which satisfy the following conditions:
\begin{itemize}
\item[$(a)$] $0\leq \phi_{j,k}\leq 1$ and $0\leq \widetilde{\phi}_{j,k}\leq 1$;
\item[$(b)$] $\phi_{j,k}=1$ on $\overline{O_{j,k}}$ and $\supp\phi_{j,k}\subseteq O'_{j,k}$;
\item[$(c)$] $\widetilde{\phi}_{j,k}=1$ on $\overline{V}_{j,k}\backslash B(0,2)$ and $\supp\widetilde{\phi}_{j,k}\subseteq V'_{j,k}\backslash \overline{B(0,1)}$.
\end{itemize}

\begin{proposition}\label{pro-for-top-isk111}
Let $L$ be a closed conic subset of $U\times(\RR^n\backslash\{0\})$ satisfying $L^c\neq \emptyset$. Let $O_{j,k}$, $O'_{j,k}$, $O''_{j,k}$, $V_{j,k}$, $V'_{j,k}$, $V''_{j,k}$, $\phi_{j,k}$ and $\widetilde{\phi}_{j,k}$ be as above. Then the mapping
\begin{gather*}
\mathcal{I}:\DD'^E_L(U)\rightarrow \DD'(U)\times (\mathcal{F}E^{\ZZ_+\times \ZZ_+}),\quad \mathcal{I}(u)=(u,\mathbf{f}_u),\,\, \mbox{where}\\
\mathbf{f}_u(j,k):=\widetilde{\phi}_{j,k}\mathcal{F}(\phi_{j,k} u),\, j,k\in\ZZ_+,
\end{gather*}
is a well-defined topological imbedding with closed image. Consequently, $\DD'^E_L(U)$ is complete.
\end{proposition}

\begin{proof}
It is straightforward to show that $\mathcal{I}$ is a well-defined continuous injection. We prove that $\mathcal{I}$ is an open mapping onto its image. It suffices to show that for each $\varphi\in\DD(U)\backslash\{0\}$ and $\psi\in\XX_0(\RR^m)\backslash\{0\}$ satisfying $(\supp \varphi \times\supp\psi)\cap L=\emptyset$, there is a finite $J\subseteq \ZZ_+\times \ZZ_+$, a continuous seminorm $\mathfrak{p}$ on $\DD'(U)$ and $C>0$ such that
\begin{equation}\label{ine-sem-for-openmap111}
\mathfrak{p}_{r;\varphi,\psi}(u)\leq C\mathfrak{p}(u)+C\sum_{(j,k)\in J}\|\widetilde{\phi}_{j,k}\mathcal{F}(\phi_{j,k} u)\|_{\mathcal{F}E},\quad u\in\DD'^E_L(U).
\end{equation}
Fix such $\varphi$ and $\psi$ and set $V:=\{0\}\cup\RR_+ \supp\psi$. Then $V$ is a closed cone in $\RR^n$ and $(\supp\varphi\times V)\cap L=\emptyset$. Pick $0<\varepsilon<1$ such that
$$
\big((\supp\varphi+\overline{B(0,\varepsilon)})\times ((V\cap \mathbb{S}^{n-1})+\overline{B(0,\varepsilon)})\big)\cap L=\emptyset\quad \mbox{and}\quad \supp\varphi+\overline{B(0,\varepsilon)}\subseteq U.
$$
Define the closed cones $V'$ and $V''$ by
$$
V'':=\RR_+\big((V\cap \mathbb{S}^{n-1})+\overline{B(0,\varepsilon)}\big)\cup\{0\}\quad \mbox{and}\quad V':=\RR_+\big((V\cap \mathbb{S}^{n-1})+\overline{B(0,\varepsilon/15)}\big)\cup\{0\}
$$
and the compact sets $K''$, $K'$ and $K$ by
$$
K'':=\supp\varphi+\overline{B(0,\varepsilon)},\quad K':=\supp\varphi+\overline{B(0,\varepsilon/15)}\quad \mbox{and}\quad K:=\supp\varphi.
$$
Of course, $K\times V\subseteq K'\times V'\subseteq K''\times V''\subseteq U\times V''$ and $(K''\times V'')\cap L=\emptyset$. In the proof of \cite[Proposition 3.7]{PP}, it is shown that there is a finite $J_0\subseteq \ZZ_+\times\ZZ_+$ such that
$$
K\times (V\backslash\{0\})\subseteq \bigcup_{(j,k)\in J_0} O_{j,k}\times V_{j,k}\subseteq K'\times V',
$$
and for each $(j,k), (j',k')\in J_0$ there is $m=m(j,k,j',k')\in\ZZ_+$ such that $O'_{j',k'}\subseteq O_{m,1}$ and $V'_{j,k}\subseteq V_{m,2}$. Set $\phi:=\sum_{(j,k)\in J_0}\phi_{j,k}\in \DD(\bigcup_{(j,k)\in J_0} O'_{j,k})$ and notice that $\phi\geq 1$ on $\bigcup_{(j,k)\in J_0}\overline{O_{j,k}}$. Hence $\varphi/\phi\in \DD(U)$. We also denote $\psi_{j,k}:=(\varphi/\phi)\phi_{j,k}\in \DD(O'_{j,k})$, $(j,k)\in J_0$, and notice that $\sum_{(j,k)\in J_0}\psi_{j,k}=\varphi$. Pick $\chi\in\XX_0(\RR^n)$ such that $0\leq \chi\leq 1$, $\chi=0$ on $B(0,3)$ and $\chi=1$ on $\RR^n\backslash B(0,4)$. Set $\widetilde{\phi}:=\sum_{(j,k)\in J_0}\widetilde{\phi}_{j,k}\in \XX_0(\RR^n)$ and notice that $\widetilde{\phi}\geq 1$ on $(\bigcup_{(j,k)\in J_0}\overline{V_{j,k}})\backslash B(0,2)$. Consequently, $(\chi\psi)/\widetilde{\phi}\in \XX_0(\RR^n)$, $\widetilde{\psi}_{j,k}:=((\chi\psi)/\widetilde{\phi})\widetilde{\phi}_{j,k}\in \XX_0(\RR^n)$, $(j,k)\in J_0$, and $\sum_{(j,k)\in J_0}\widetilde{\psi}_{j,k}=\chi\psi$. We infer
\begin{equation}\label{sum-for-bou-ftoplimbed111}
\psi\mathcal{F}(\varphi u)= (1-\chi)\psi\mathcal{F}(\varphi u)+\sum_{\substack{(j,k)\in J_0\\ (j',k')\in J_0}} \widetilde{\psi}_{j,k}\mathcal{F}(\psi_{j',k'} u).
\end{equation}
Since $\DD'(U)\rightarrow \DD(U)$, $v\mapsto (1-\chi)\psi\mathcal{F}(\varphi v)$, is well-defined and continuous, $v\mapsto \|(1-\chi)\psi\mathcal{F}(\varphi v)\|_{\mathcal{F}E}$ is a continuous seminorm on $\DD'(U)$. It remains to show that for each $(j,k),(j',k')\in J_0$, $\|\widetilde{\psi}_{j,k}\mathcal{F}(\psi_{j',k'} u)\|_{\mathcal{F}E}$ can be bounded by seminorms as in the right-hand side of \eqref{ine-sem-for-openmap111}. Let $(j,k),(j',k')\in J_0$ be arbitrary but fixed. In view of the above, we can find $m\in\ZZ_+$ such that $O'_{j',k'}\subseteq O_{m,1}$ and $V'_{j,k}\subseteq V_{m,2}$. Hence $\psi_{j',k'}=\psi_{j',k'}\phi_{m,1}$ and $\widetilde{\psi}_{j,k}=\widetilde{\psi}_{j,k}\widetilde{\phi}_{m,2}$ and thus
\begin{multline*}
\widetilde{\psi}_{j,k}\mathcal{F}(\psi_{j',k'} u) =(2\pi)^{-n}\widetilde{\psi}_{j,k}\widetilde{\phi}_{m,2}\left(\mathcal{F}\psi_{j',k'}*\left(\widetilde{\phi}_{m,1}\mathcal{F}(\phi_{m,1} u)\right)\right)\\
+(2\pi)^{-n}\widetilde{\psi}_{j,k}\widetilde{\phi}_{m,2}\left(\mathcal{F}\psi_{j',k'}*\left((1-\widetilde{\phi}_{m,1})\mathcal{F}(\phi_{m,1} u)\right)\right).
\end{multline*}
As $\DD'(U)\rightarrow \displaystyle \lim_{\substack{\longrightarrow\\ s\rightarrow\infty}}L^1_{\langle\cdot\rangle^{-s}}(\RR^n)$, $v\mapsto \mathcal{F}(\phi_{m,1} v)$, is well-defined and continuous, Lemma \ref{lem-for-est-oftheremininestwavefrps} implies that $u\mapsto \|(\mbox{second term})\|_{\mathcal{F}E}$ is a well-defined continuous seminorm on $\DD'(U)$ of $u$ (apply the lemma with $\psi_2:=\widetilde{\psi}_{j,k}\widetilde{\phi}_{m,2}$ and $\psi_1:=(1-\widetilde{\phi}_{m,1})$). For the first term, we infer
\begin{align*}
\left\|\widetilde{\psi}_{j,k}\widetilde{\phi}_{m,2}\left(\mathcal{F}\psi_{j',k'}*\left(\widetilde{\phi}_{m,1}\mathcal{F}(\phi_{m,1} u)\right)\right)\right\|_{\mathcal{F}E}&\leq C\left\|\mathcal{F}\psi_{j',k'}*\left(\widetilde{\phi}_{m,1}\mathcal{F}(\phi_{m,1} u)\right)\right\|_{\mathcal{F}E}\\
&\leq C\|\mathcal{F}\psi_{j',k'}\|_{L^1_{\omega_{\mathcal{F}E}}(\RR^n)}\|\widetilde{\phi}_{m,1}\mathcal{F}(\phi_{m,1} u)\|_{\mathcal{F}E},
\end{align*}
which completes the proof for the fact that $\mathcal{I}$ is a topological imbedding. The proof that the range of $\mathcal{I}$ is closed is analogous to the proof of the corresponding fact in \cite[Proposition 3.7]{PP} (cf. \cite[Remark 3.8]{PP}) and we omit it.
\end{proof}

Besides the completeness, one can derive other interesting properties of $\DD'^E_L(U)$ from Proposition \ref{pro-for-top-isk111} as in \cite{PP}; however, we will not need such facts.\\
\indent Next, we show that $\DD(U)$ is dense in $\DD'^E_L(U)$ when $E$ is a (TMIB)-space. This is verified by the next result whose proof is similar to \cite[Proposition 3.13]{PP}.

\begin{proposition}\label{seq-den-comsmf}
Let $\phi_j\in\DD(U)$, $j\in\ZZ_+$, be such that $0\leq \phi_j\leq 1$, $\phi_j=1$ on the compact $\{x\in U\,|\, |x|\leq j,\, \operatorname{dist}(x,\partial U)\geq 3/j\}$ and $\supp\phi_j\subseteq \{x\in U\,|\, \operatorname{dist}(x,\partial U)> 2/j\}$. Let $\chi$ be a nonnegative function in $\DD(\RR^n)$ satisfying $\supp\chi\subseteq \{x\in\RR^n\,|\, |x|\leq 1\}$ and $\int_{\RR^n} \chi(x) dx=1$ and set $\chi_j(x):=j^n\chi(jx)$, $x\in\RR^n$, $j\in\ZZ_+$. For each $j\in\ZZ_+$, the operators
\begin{equation}\label{ope-app-ide-simc}
P_j:\DD'(U)\rightarrow\DD(U),\quad P_ju=\chi_j*(\phi_j u),
\end{equation}
are well-defined and continuous. Furthermore, for every (TMIB)-space $E$ which satisfies $(iv)$ and every closed conic subset $L$ of $U\times (\RR^n\backslash\{0\})$, $\{P_j\}_{j\in\ZZ_+}$ is a bounded subset of $\mathcal{L}_b(\DD'^E_L(U))$ and $P_j\rightarrow \operatorname{Id}$ in $\mathcal{L}_{\sigma}(\DD'^E_L(U))$. In particular, $\DD(U)$ is sequentially dense in $\DD'^E_L(U)$.
\end{proposition}

\begin{proof}
We point out that $P_j\rightarrow\operatorname{Id}$ in $\mathcal{L}_b(\DD'(U))$; this can be shown in the same way as in \cite[Proposition 3.13]{PP}. This proves the proposition when $L=U\times (\RR^n\backslash\{0\})$.\\
\indent Assume that $L^c\neq\emptyset$. In view of the above, to show the boundedness of $\{P_j\}_{j\in\ZZ_+}$ in $\mathcal{L}_b(\DD'^E_L(U))$ it remains to prove that for every bounded set $B$ in $\DD'^E_L(U)$ and every seminorm $\mathfrak{p}_{E;\varphi,\psi}$, it holds that $\sup_{j\in\ZZ_+}\sup_{u\in B}\mathfrak{p}_{E;\varphi,\psi}(P_j u)<\infty$. Let $B$ be a bounded subset of $\DD'^E_L(U)$ and let $\varphi\in\DD(U)\backslash\{0\}$ and $\psi\in\XX_0(\RR^n)\backslash\{0\}$ be such that $(\supp\varphi\times\supp\psi)\cap L=\emptyset$. Set $V:=\{0\}\cup\RR_+\supp\psi$. Then $V$ is a closed cone in $\RR^n$ and $(\supp\varphi\times V)\cap L=\emptyset$. A standard compactness argument shows that there are $\widetilde{\varphi}\in\DD(U)$ and closed cones $V'$ and $\widetilde{V}$ in $\RR^n$ such that $\widetilde{\varphi}=1$ on a neighbourhood of $\supp\varphi$, $V\backslash\{0\}\subseteq\operatorname{int}V'$, $V'\backslash\{0\}\subseteq \operatorname{int}\widetilde{V}$ and $(\supp\widetilde{\varphi}\times\widetilde{V})\cap L=\emptyset$. Pick $\widetilde{\psi}\in \XX_0(\RR^n)$ such that $\supp\widetilde{\psi}\subseteq \widetilde{V}$ and $\widetilde{\psi}=1$ on $V'\backslash B(0,1)$. There is $j_0\in\ZZ_+$ such that $u_j:=\chi_j*(\widetilde{\varphi} u)\in\DD(U)$, for all $u\in B$, $j\geq j_0$, and there is $j'\geq j_0$ such that $\phi_j=\widetilde{\varphi}=1$ on $\supp\varphi +\overline{B(0,2/j)}$, for all $j\geq j'$. Whence,
$$
\varphi P_ju-\varphi u_j=\varphi(\chi_j*((\phi_j-\widetilde{\varphi})u))=0,\quad j\geq j',\,\, u\in B.
$$
Thus, it suffices to show $\sup_{j\geq j'}\sup_{u\in B}\mathfrak{p}_{E;\varphi,\psi}(u_j)<\infty$. Write
\begin{equation*}
\psi\mathcal{F}(\varphi u_j)= (2\pi)^{-n}\psi\left(\mathcal{F}\varphi*\left(\widetilde{\psi}\mathcal{F}\chi_j\mathcal{F}(\widetilde{\varphi}u)\right)\right) +(2\pi)^{-n}\psi\left(\mathcal{F}\varphi*\left((1-\widetilde{\psi})\mathcal{F}\chi_j\mathcal{F}(\widetilde{\varphi}u)\right)\right)
\end{equation*}
Since $\mathcal{F}\chi_j=\mathcal{F}\chi(\cdot/j)$ and there is $l>0$ such that $\sup_{u\in B}\|\langle\cdot\rangle^{-l}\mathcal{F}(\widetilde{\varphi}u)\|_{L^{\infty}(\RR^n)}<\infty$, $\{\mathcal{F}\chi_j\mathcal{F}(\widetilde{\varphi}u)\,|\, u\in B,\, j\in\ZZ_+\}$ is a bounded subset of $\displaystyle \lim_{\substack{\longrightarrow\\ s\rightarrow\infty}}L^1_{\langle\cdot\rangle^{-s}}(\RR^n)$. Hence, Lemma \ref{lem-for-est-oftheremininestwavefrps} implies that the $\mathcal{F}E$-norm of the second term is uniformly bounded for $j\in\ZZ_+$ and $u\in B$. For the first term, we infer
\begin{align*}
\left\|\psi\left(\mathcal{F}\varphi*\left(\widetilde{\psi}\mathcal{F}\chi_j\mathcal{F}(\widetilde{\varphi}u)\right)\right)\right\|_{\mathcal{F}E} \leq C_1\|\mathcal{F}\varphi\|_{L^1_{\omega_{\mathcal{F}E}}(\RR^n)} \|\mathcal{F}\chi_j\|_{\mathcal{F}L^1_{\nu_{\mathcal{F}E}}(\RR^n)}\|\widetilde{\psi}\mathcal{F}(\widetilde{\varphi}u)\|_{\mathcal{F}E}.
\end{align*}
Since $\sup_{j\in\ZZ_+}\|\mathcal{F}\chi_j\|_{\mathcal{F}L^1_{\nu_{\mathcal{F}E}}(\RR^n)}=\sup_{j\in\ZZ_+}\|\chi_j\|_{L^1_{\nu_{\mathcal{F}E}}(\RR^n)} <\infty$, the right-hand side is uniformly bounded for $j\in\ZZ_+$ and $u\in B$. This completes the proof of the fact that $\{P_j\}_{j\in\ZZ_+}$ is bounded in $\mathcal{L}_b(\DD'^E_L(U))$.\\
\indent We now show that $P_j\rightarrow \operatorname{Id}$ in $\mathcal{L}_{\sigma}(\DD'^E_L(U))$. Since the convergence holds in $\mathcal{L}_{\sigma}(\DD'(U))$, it remains to show that for each $u\in\DD'^E_L(U)$ and $\varphi\in\DD(U)\backslash\{0\}$ and $\psi\in\XX_0(\RR^m)\backslash\{0\}$ satisfying $(\supp\times\supp\psi)\cap L=\emptyset$, it holds that $\mathfrak{p}_{E;\varphi,\psi}(P_ju-u)\rightarrow0$, as $j\rightarrow\infty$. Fix such $u$, $\varphi$ and $\psi$. Let $\widetilde{\varphi}$, $V$, $V'$, $\widetilde{V}$, $\widetilde{\psi}$ and $j'$ be as above. Again, we write $u_j:=\chi_j*(\widetilde{\varphi}u)$ and recall that $\varphi P_ju=\varphi u_j$, $j\geq j'$. As before, write
\begin{multline}\label{equ-con-infeofdels}
\psi\mathcal{F}(\varphi P_ju)-\psi\mathcal{F}(\varphi u) =(2\pi)^{-n}\psi\left(\mathcal{F}\varphi*\left(\widetilde{\psi}(\mathcal{F}\chi_j-1)\mathcal{F}(\widetilde{\varphi}u)\right)\right)\\
+(2\pi)^{-n}\psi\left(\mathcal{F}\varphi*\left((1-\widetilde{\psi})(\mathcal{F}\chi_j-1)\mathcal{F}(\widetilde{\varphi}u)\right)\right).
\end{multline}
There is $l>0$ such that $\|\langle\cdot\rangle^{-l}\mathcal{F}(\widetilde{\varphi}u)\|_{L^{\infty}(\RR^n)}<\infty$. Since $\mathcal{F}\chi_j=\mathcal{F}\chi(\cdot/j)$ and $\mathcal{F}\chi(0)=1$, dominated convergence implies that $(\mathcal{F}\chi_j-1)\mathcal{F}(\widetilde{\varphi}u)\rightarrow0$ in $L^1_{\langle\cdot\rangle^{-l-n-1}}(\RR^n)$ and Lemma \ref{lem-for-est-oftheremininestwavefrps} verifies that the second term in \eqref{equ-con-infeofdels} converges to $0$ in $\mathcal{F}E$. For the first term, since $\mathcal{F}^{-1}(\widetilde{\psi}\mathcal{F}(\widetilde{\varphi}u))\in E$, we infer
\begin{align*}
&\left\|\psi\left(\mathcal{F}\varphi*\left(\widetilde{\psi}(\mathcal{F}\chi_j-1) \mathcal{F}(\widetilde{\varphi}u)\right)\right)\right\|_{\mathcal{F}E}\\
&\leq C_2\|\mathcal{F}\varphi\|_{L^1_{\omega_{\mathcal{F}E}}(\RR^n)} \|(\mathcal{F}\chi_j-1)\widetilde{\psi}\mathcal{F}(\widetilde{\varphi}u)\|_{\mathcal{F}E}\\
&=C_2\|\mathcal{F}\varphi\|_{L^1_{\omega_{\mathcal{F}E}}(\RR^n)} \|\chi_j*\mathcal{F}^{-1}(\widetilde{\psi}\mathcal{F}(\widetilde{\varphi}u))- \mathcal{F}^{-1}(\widetilde{\psi}\mathcal{F}(\widetilde{\varphi}u))\|_E\rightarrow0,\quad \mbox{as}\quad j\rightarrow\infty,
\end{align*}
in view of \cite[Corollary 1]{DPV}. This completes the proof of the proposition.
\end{proof}

\section{The pullback on $\DD'^E_L$ by smooth maps of constant rank}\label{Sec2}

We are now ready to show our main results on the pullback by smooth maps. We first study the pullback by local diffeomorphisms on $\DD'^E_L$. For this purposes, we consider the following additional condition:
\begin{itemize}
\item[$(v)$] $E$ is locally diffeomorphism invariant, namely for every open set $O$, every diffeomorphism $f:O\rightarrow f(O)$ and every compact set $K$ in $f(O)$, the pullback $f^*:\EE'(f(O))\rightarrow\EE'(O)(\subseteq \EE'(\RR^n))$ restricts to a continuous map $f^*:E_K\rightarrow E$, where $E_K$ is the Banach subspace of $E$ containing the elements in $E$ supported by $K$.
\end{itemize}

\begin{lemma}\label{lem-for-diffinvorskl}
The (TMIB)-space $E$ satisfies $(v)$ if and only if $E'$ satisfies $(v)$.
\end{lemma}

\begin{proof}
It is straightforward to show that $E$ satisfying $(v)$ implies that $E'$ does so. Assume that $E'$ satisfies $(v)$. Let $K$ be a compact set in $f(O)$. It is enough to show that $f^*(E_K)\subseteq E$ since this automatically implies $f^*(E_K)\subseteq E_{f^{-1}(K)}$ and the closed graph theorem implies that continuity of the map. Let $e\in E_K$ and let $\chi_j\in\DD(\RR^n)$, $j\in\ZZ_+$, be as in the statement of Proposition \ref{seq-den-comsmf}. There is $j_0\in\ZZ_+$ such that $\supp\chi_j*e\subseteq K'$, $j\geq j_0$, for some fixed compact set $K'\supseteq K$ in $f(O)$. Pick $\phi\in \DD(f(O))$ such that $\phi=1$ on a neighbourhood of $K'$. For $e'\in E'$, we infer
\begin{align*}
|\langle e',f^*(\chi_j*e)-f^*(\chi_k*e)\rangle|&=|\langle f^{-1\,*}(e'f^*(|f^{-1\,'}|\phi)),\chi_j*e-\chi_k*e\rangle|\\
&\leq \|f^{-1\,*}(e'f^*(|f^{-1\,'}|\phi))\|_{E'}\|\chi_j*e-\chi_k*e\|_E\\
&\leq C\|e'\|_{E'}\|\chi_j*e-\chi_k*e\|_E,
\end{align*}
whence $\|f^*(\chi_j*e)-f^*(\chi_k*e)\|_E\leq C\|\chi_j*e-\chi_k*e\|_E$. In view of \cite[Corollary 1]{DPV}, $\chi_j*e\rightarrow e$ in $E$ and consequently $\{f^*(\chi_j*e)\}_{j\in\ZZ_+}$ is a Cauchy sequence in $E$. Since it converges to $f^*e\in \EE'(O)$ in $\EE'(O)$, we deduce $f^*e\in E$ and the proof is complete.
\end{proof}

We point out that the $L^p$-Sobolev spaces $W^{r,p}(\RR^m)$, $1<p<\infty$, of order $r\in \RR$, satisfy the condition $(v)$ (this is a nontrivial but classical result and it essentially follows from \cite[Theorem 18.1.17, p. 81]{hor2}).\\
\indent Before we state the next result, we make the following observation. Let $f:O\rightarrow U$ be a local diffeomorphism with $O$ and $U$ open sets in $\RR^m$. Then the pullback $f^*:\mathcal{C}^{\infty}(U)\rightarrow \mathcal{C}^{\infty}(O)$, $f^*(u)=u\circ f$, uniquely extends to a well-defined and continuous mapping $f^*:\DD'(U)\rightarrow \DD'(O)$. This follows from \cite[Theorem 3.21]{PP} (cf. \cite[Remark 3.6]{PP}) by applying it with $L=U\times(\RR^n\backslash\{0\})$ since $\mathcal{N}_f=f(O)\times \{0\}$.

\begin{theorem}\label{pull-back-diff}
Let $O$ and $U$ be open sets in $\RR^m$, let $f:O\rightarrow U$ be a local diffeomorphism and let $L$ be a closed conic subset of $U\times (\RR^m\backslash\{0\})$. Let $E$ be a (TMIB) or a (DTMIB)-space which satisfies $(iv)$ and $(v)$. The pullback $f^*:\mathcal{C}^{\infty}(U)\rightarrow \mathcal{C}^{\infty}(O)$, $f^*(u)=u\circ f$, uniquely extends to a well-defined and continuous map $f^*:\DD'(U)\rightarrow \DD'(O)$ and the latter restricts to a well-defined and continuous map $f^*:\DD'^E_L(U)\rightarrow \DD'^E_{f^*L}(O)$. Consequently, if $f$ is a diffeomorphism, then $f^*:\DD'^E_L(U)\rightarrow \DD'^E_{f^*L}(O)$ is a topological isomorphism.
\end{theorem}

\begin{proof}
We already pointed out that the pullback extends to a well-defined and continuous map $f^*:\DD'(U)\rightarrow \DD'(O)$. To prove the second part, for each $y\in U$, set $L_y:=\{\eta\in\RR^m\backslash\{0\}\,|\, (y,\eta)\in L\}$; $L_y$ is a (possibly empty) closed cone in $\RR^m\backslash\{0\}$. Notice that $\mathcal{N}_f=f(O)\times \{0\}$ and thus $\mathcal{N}_f\cap L=\emptyset$. Arguing as in the beginning of the proof of \cite[Theorem 3.21]{PP}, one shows that this implies that for every $x_0\in O$ and every open cone $G$ in $\RR^m\backslash\{0\}$ satisfying ${}^tf'(x_0)L_{f(x_0)}\subseteq G$ there are a relatively compact open set $O_0\ni x_0$ satisfying $\overline{O_0}\subseteq O$, closed cones $V'$ and $V$ in $\RR^m\backslash\{0\}$ satisfying $L_{f(x_0)}\subseteq \operatorname{int} V'\subseteq V' \subseteq \operatorname{int}V$ and ${}^tf'(x_0)V\subseteq G$ and an open neighbourhood $U_0\subseteq U$ of $f(x_0)$ which satisfy the following conditions: $f(\overline{O_0})\subseteq U_0$, $\bigcup_{y\in U_0} L_y \subseteq \operatorname{int}V'$, ${}^tf'(x)\eta\in G$ for all $x\in \overline{O_0}$, $\eta\in V$, and
\begin{equation}\label{ine-for-stati-phase-methd}
|{}^t f'(x)\eta-\xi|\geq \varepsilon(|\xi|+|\eta|),\quad x\in\overline{O_0},\, \xi\in \RR^m\backslash G,\, \eta\in V\cup\{0\},
\end{equation}
for some small enough $\varepsilon>0$. Since $f$ is a local diffeomorphism, we can take $O_0$ to be small enough so that $f$ is a diffeomorphism from an open neighbourhood of $\overline{O_0}$ onto an open subset of $U_0$.\\
\indent To show that $f^*$ restricts to a well-defined and continuous map $f^*:\DD'^E_L(U)\rightarrow \DD'^E_{f^*L}(O)$, it is enough to prove that for each $\varphi\in \DD(O)\backslash\{0\}$ and $\chi\in\XX_0(\RR^m)$ satisfying $(\supp\varphi\times\supp\chi)\cap f^*L=\emptyset$, $\chi\mathcal{F}(\varphi f^*u)\in\mathcal{F}E$ when $u\in\DD'^E_L(U)$ and to bound $\mathfrak{p}_{E;\varphi,\chi}(f^*u)=\|\chi\mathcal{F}(\varphi f^*u)\|_{\mathcal{F}E}$ (the first part verifies that for each continuous seminorm $\mathfrak{p}$ on $\DD'(O)$, $\mathfrak{p}(f^*u)$ is bounded by a continuous seminorm on $\DD'(U)$ of $u$). Let $\varphi\in \DD(O)\backslash\{0\}$ and $\chi\in\XX_0(\RR^m)$ satisfy $(\supp\varphi\times\supp\chi)\cap f^*L=\emptyset$. Then $G_1:=\{0\}\cup\RR_+\supp\chi$ is a closed cone in $\RR^m$ that satisfies $(\supp\varphi\times G_1)\cap f^*L=\emptyset$ (when $f^*L=O\times(\RR^m\backslash\{0\})$, we take $\chi=0$ and thus $G_1=\{0\}$). The cone $G:=\RR^m\backslash G_1$ is open in $\RR^m\backslash\{0\}$ and satisfies ${}^tf'(x)L_{f(x)}\subseteq G$, $x\in\supp\varphi$. We apply the above construction for this $G$ and each $x\in \supp\varphi$ to obtain the open neighbourhoods $O_x$ and $U_x$ of $x$ and $f(x)$ respectively having the above properties. As $\supp\varphi$ is compact, there are finitely many such $O_j$, $j=1,\ldots,l$, whose union covers $\supp\varphi$. We denote by $U_j$, $j=1,\ldots,l$, the corresponding subsets of $U$ and by $V'_j$ and $V_j$, $j=1,\ldots,l$, the corresponding closed cones in $\RR^m\backslash\{0\}$ in the above construction. Let $\psi_j\in\DD(O_j)$, $0\leq \psi_j\leq 1$, $j=1,\ldots, l$, be such that $\sum_{j=1}^l\psi_j=1$ on a neighbourhood of $\supp\varphi$. Pick $\phi_j\in\DD(U_j)$ such that $\phi_j=1$ on a neighbourhood of $f(\overline{O_j})$, $j=1,\ldots,l$. Choose $\widetilde{\chi}_j\in\XX_0(\RR^m)$ such that $\supp\widetilde{\chi}_j\subseteq \RR^m\backslash V'_j$ and $\widetilde{\chi}_j=1$ on $\{\eta\in\RR^m\backslash V_j\,|\, |\eta|>1\}$. For $u\in\mathcal{C}^{\infty}(U)$, as in the proof of \cite[Theorem 3.21]{PP}, write
\begin{equation}\label{ide-for-fourtransofpbfks}
\mathcal{F}(\varphi f^*u)(\xi)=\frac{1}{(2\pi)^m}\sum_{j=1}^l\int_{\RR^m}\mathcal{F}(\phi_j u)(\eta)\widetilde{I}_{\psi_j\varphi}(\xi,\eta)d\eta,\quad \xi\in\RR^m,
\end{equation}
where, $\widetilde{I}_{\psi_j\varphi}$ stands for the function
\begin{equation*}
\widetilde{I}_{\psi_j\varphi}:\RR^{2m}\rightarrow \CC,\quad \widetilde{I}_{\psi_j\varphi}(\xi,\eta):=\int_O e^{i (f(x)\eta-x\xi)} \psi_j(x)\varphi(x)dx;
\end{equation*}
clearly $\widetilde{I}_{\psi_j\varphi}\in\DD_{L^{\infty}}(\RR^{2m})$. We need the following facts for $\widetilde{I}_{\psi_j\varphi}$, $j=1,\ldots,l$.\\
\\
\noindent \textbf{Claim 1} For every $k\in\NN$ and $\alpha,\beta\in\NN^m$,
\begin{align}
&\sup_{x,\xi\in\RR^m}\langle\xi\rangle^{-2k}\langle\eta\rangle^{2k} |\partial^{\alpha}_{\xi}\partial^{\beta}_{\eta}\widetilde{I}_{\psi_j\varphi}(\xi,\eta)|<\infty,\quad j=1,\ldots,l;\label{est-cla-forskls1}\\
&\sup_{x,\xi\in\RR^m}\langle\eta\rangle^{-2k}\langle\xi\rangle^{2k} |\partial^{\alpha}_{\xi}\partial^{\beta}_{\eta}\widetilde{I}_{\psi_j\varphi}(\xi,\eta)|<\infty,\quad j=1,\ldots,l.\label{est-cla-forskls2}
\end{align}
\\
\indent We defer its proof for later and continue with the proof of the theorem. This fact implies that even when $u\in\DD'(U)$, the integrals on the right-hand side of \eqref{ide-for-fourtransofpbfks} are absolutely convergent uniformly for $\xi$ in a compact set and each of them gives a continuous function of $\xi$ with polynomial growth. Since $\mathcal{C}^{\infty}(U)$ is sequentially dense $\DD'(U)$, we deduce that \eqref{ide-for-fourtransofpbfks} is valid for all $u\in\DD'(U)$: indeed, for a sequence $\{u_k\}_{k\in\ZZ_+}\subseteq \mathcal{C}^{\infty}(U)$ converging to $u\in\DD'(U)$, $\mathcal{F}(\varphi f^*u_k)\rightarrow\mathcal{F}(\varphi f^*u)$ in view of the continuity of $f^*:\DD'(U)\rightarrow\DD'(O)$ while the integrals in \eqref{ide-for-fourtransofpbfks} with $u_k$ in place of $u$ converge because of the dominated convergence theorem in view of \eqref{est-cla-forskls1}. Let $u\in\DD'^E_L(U)$. Write
\begin{equation}\label{ine-for-bound-semin-osk111}
\mathcal{F}(\varphi f^*u)(\xi)=\frac{1}{(2\pi)^m}\sum_{j=1}^l(I_{1;j}(\xi)+I_{2;j}(\xi)),\quad \xi\in\RR^m,
\end{equation}
with
\begin{align*}
I_{1;j}(\xi)&:=\int_{\RR^m}(1-\widetilde{\chi}_j(\eta))\mathcal{F}(\phi_j u)(\eta)\widetilde I_{\psi_j\varphi}(\xi,\eta)d\eta,\\
I_{2;j}(\xi)&:=\int_{\RR^m}\widetilde{\chi}_j(\eta)\mathcal{F}(\phi_j u)(\eta)\widetilde{I}_{\psi_j\varphi}(\xi,\eta)d\eta;
\end{align*}
of course, $I_{1;j}$ and $I_{2;j}$ are smooth functions with polynomial growth. Our goal is to show $\chi\mathcal{F}(\varphi f^*u)\in\mathcal{F}E$ and to estimate $\mathfrak{p}_{E;\varphi,\chi}(f^*u)=\|\chi\mathcal{F}(\varphi f^*u)\|_{\mathcal{F}E}$. The stationary phase method \cite[Theorem 7.7.1, p. 216]{hor} together with \eqref{ine-for-stati-phase-methd} implies that for every $\alpha\in\NN^m$ and $N>0$ there is $C_{N,\alpha}>0$ such that $|\widetilde{I}_{\psi_j\varphi}(\xi,\eta)|\leq C_{N,\alpha}(1+|\xi|+|\eta|)^{-N}$, $\xi\in G_1$, $\eta\in V_j$. Consequently,
\begin{equation*}
\DD'(U)\rightarrow\SSS(\RR^m),\quad v\mapsto \chi\int_{\RR^m}(1-\widetilde{\chi}_j(\eta))\mathcal{F}(\phi_j v)(\eta)\widetilde I_{\psi_j\varphi}(\cdot,\eta)d\eta,
\end{equation*}
is well-defined and maps bounded sets into bounded sets (cf. \eqref{est-cla-forskls2}). Since $\DD'(U)$ is bornological, the map is continuous; whence $\|\chi I_{1;j}\|_{\mathcal{F}E}$ is a continuous seminorm on $\DD'(U)$ of $u$. In view of \eqref{ine-for-bound-semin-osk111}, it remains to show that $\chi I_{2;j}\in\mathcal{F}E$ and to bound $\|\chi I_{2;j}\|_{\mathcal{F}E}$. We first consider the case when $E$ is a (DTMIB)-space; i.e. $E=E_0'$ with $E_0$ a (TMIB)-space which satisfies $(iv)$ and $(v)$ in view of Lemma \ref{lem-for-tmbsjk} and Lemma \ref{lem-for-diffinvorskl}. For $\vartheta\in \SSS(\RR^m)$, we infer (cf. \eqref{est-cla-forskls1})
\begin{equation*}
\langle I_{2;j},\vartheta\rangle=\iint_{\RR^{2m}}\widetilde{\chi}_j(\eta)\mathcal{F}(\phi_j u)(\eta)\widetilde{I}_{\psi_j\varphi}(\xi,\eta)\vartheta(\xi)d\xi d\eta=\langle \widetilde{\chi}_j\mathcal{F}(\phi_j u),v\rangle
\end{equation*}
where $v(\eta):=\int_{\RR^m}\widetilde{I}_{\psi_j\varphi}(\xi,\eta)\vartheta(\xi)d\xi$ and $v\in\SSS(\RR^m)$ in view of \eqref{est-cla-forskls1}. Notice that
\begin{align*}
v(\eta)&=\langle \vartheta,\mathcal{F}(e^{if(\cdot)\eta}\psi_j\varphi)\rangle =\langle\psi_j\varphi\mathcal{F}\vartheta,e^{if(\cdot)\eta}\rangle=\left\langle f^{-1\,*}\left((|f^{-1\,'}|\circ f)\psi_j\varphi\mathcal{F}\vartheta\right),e^{i\,\cdot\,\eta}\right\rangle\\
&= (2\pi)^m\mathcal{F}^{-1}\left(f^{-1\,*}\left((|f^{-1\, '}|\circ f)\psi_j\varphi\mathcal{F}\vartheta\right)\right)(\eta).
\end{align*}
The property $(v)$ for $E_0$ implies $\|v\|_{\mathcal{F}^{-1}E_0}\leq C'\|\mathcal{F}\vartheta\|_{E_0}=C'\|\vartheta\|_{\mathcal{F}^{-1}E_0}$. Since $(\mathcal{F}^{-1}E_0)'=\mathcal{F}E$ isometrically and $U_j\times(\RR^m\backslash V'_j)\cap L=\emptyset$, we have
\begin{equation*}
|\langle I_{2;j},\vartheta\rangle|\leq \|\widetilde{\chi}_j\mathcal{F}(\phi_j u)\|_{\mathcal{F}E}\|v\|_{\mathcal{F}^{-1}E_0}\leq C'\|\widetilde{\chi}_j\mathcal{F}(\phi_j u)\|_{\mathcal{F}E}\|\vartheta\|_{\mathcal{F}^{-1}E_0}.
\end{equation*}
As $\SSS(\RR^m)$ is dense in $\mathcal{F}^{-1}E_0$, we deduce $I_{2;j}\in \mathcal{F}E$ and $\|I_{2;j}\|_{\mathcal{F}E}\leq C'\|\widetilde{\chi}_j\mathcal{F}(\phi_j u)\|_{\mathcal{F}E}$; consequently, $\chi I_{2;j}\in\mathcal{F}E$ and $\|\chi I_{2;j}\|_{\mathcal{F}E}\leq C'_1\|I_{2;j}\|_{\mathcal{F}E}\leq C'C'_1\mathfrak{p}_{E;\phi_j,\widetilde{\chi}_j}(u)$.\\
\indent Assume now $E$ is a (TMIB)-space. Since $\mathcal{C}^{\infty}(U)$ is dense in $\DD'^E_L(U)$ (in view of Proposition \ref{seq-den-comsmf}), it suffices to bound $\mathfrak{p}_{E;\varphi,\chi}(f^*u)=\|\chi\mathcal{F}(\varphi f^*u)\|_{\mathcal{F}E}$ when $u\in\mathcal{C}^{\infty}(U)$ (of course, $\chi\mathcal{F}(\varphi f^*u)\in\SSS(\RR^m)$ when $u\in\mathcal{C}^{\infty}(U)$). In view of the above, we only need to show that $\chi I_{2;j}\in\mathcal{F}E$ and to bound $\|\chi I_{2;j}\|_{\mathcal{F}E}$. For this purpose, we make the following\\
\\
\noindent \textbf{Claim 2} For each $w\in\SSS'(\RR^m)$, the function $\RR^m\rightarrow \CC$, $\eta\mapsto \langle w, \widetilde{I}_{\psi_j\varphi}(\cdot,\eta)\rangle$, is continuous and with polynomial growth. When $u\in\mathcal{C}^{\infty}(U)$, $I_{2;j}\in\SSS(\RR^m)$ and
\begin{equation*}
\langle w,I_{2;j}\rangle=\int_{\RR^m}\widetilde{\chi}_j(\eta)\mathcal{F}(\phi_j u)(\eta)\langle w,\widetilde{I}_{\psi_j\varphi}(\cdot,\eta)\rangle d\eta,\quad w\in\SSS'(\RR^m).
\end{equation*}
\\
\indent We defer its proof for later and continue with the proof of the theorem. For $e'\in (\mathcal{F}E)'=\mathcal{F}^{-1}E'$, the claim implies $\langle e', I_{2;j}\rangle=\langle v,\widetilde{\chi}_j\mathcal{F}(\phi_j u)\rangle$, where $v(\eta):=\langle e',\widetilde{I}_{\psi_j\varphi}(\cdot,\eta)\rangle$, $\eta\in\RR^m$, is a continuous function with polynomial growth. Analogously as above, one shows that
\begin{equation*}
v(\eta)= (2\pi)^m\mathcal{F}^{-1}\left(f^{-1\,*}\left((|f^{-1\, '}|\circ f)\psi_j\varphi\mathcal{F}e'\right)\right)(\eta),\quad \eta\in\RR^m.
\end{equation*}
Consequently, $v\in\mathcal{F}^{-1}E'=(\mathcal{F}E)'$ and $\|v\|_{(\mathcal{F}E)'}\leq C''\|\mathcal{F}e'\|_{E'}=C''\|e'\|_{(\mathcal{F}E)'}$. We infer $|\langle e', I_{2;j}\rangle|\leq C''\|e'\|_{(\mathcal{F}E)'}\|\widetilde{\chi}_j\mathcal{F}(\phi_j u)\|_{\mathcal{F}E}$. Whence, $\|\chi I_{2;j}\|_{\mathcal{F}E}\leq C''_1\|I_{2;j}\|_{\mathcal{F}E}\leq C''_1C''\mathfrak{p}_{E;\phi_j,\widetilde{\chi}_j}(u)$ which completes the proof since $U_j\times(\RR^m\backslash V'_j)\cap L=\emptyset$.\\
\\
\noindent \textbf{Proof of Claim 1.} We only show \eqref{est-cla-forskls1} as the proof of \eqref{est-cla-forskls2} is analogous. Since $\supp\psi_j\subseteq O_j$ and $f$ restricts to a diffeomorphism on $O_j$, we can change variables to obtain
\begin{equation*}
\widetilde{I}_{\psi_j\varphi}(\xi,\eta)=\int_{f(O_j)} e^{i y\eta}e^{-if^{-1}(y)\xi} \vartheta(y)dy,\quad \mbox{with}\quad \vartheta:=|f^{-1\, '}| (\psi_j\varphi)\circ f^{-1}\in\DD(f(O_j)).
\end{equation*}
Consequently, for $k\in\NN$ and $\alpha,\beta\in\NN^m$, we infer
\begin{equation*}
\langle\eta\rangle^{2k}|\partial^{\alpha}_{\xi}\partial^{\beta}_{\eta}\widetilde{I}_{\psi_j\varphi}(\xi,\eta)|\leq\int_{f(O_j)} \left|(\operatorname{Id}-\Delta_y)^k\left(e^{-if^{-1}(y)\xi} (f^{-1}(y))^{\alpha}y^{\beta}\vartheta(y)\right)\right|dy\leq C\langle\xi\rangle^{2k},
\end{equation*}
which competes the proof of the claim.\\
\\
\noindent \textbf{Proof of Claim 2.} It is straightforward to verify that the map $\RR^m\rightarrow\SSS(\RR^m)$, $\eta\mapsto e^{if(\cdot)\eta}\psi_j\varphi$, is well-defined and continuous. Consequently, the map $\RR^m\rightarrow\SSS(\RR^m)$, $\eta\mapsto \widetilde{I}_{\psi_j\varphi}(\cdot,\eta)=\mathcal{F}(e^{if(\cdot)\eta}\psi_j\varphi)$, is well-defined and continuous, which, in turn, shows that for each $w\in\SSS'(\RR^m)$, $\eta\mapsto \langle w, \widetilde{I}_{\psi_j\varphi}(\cdot,\eta)\rangle$ is a continuous function on $\RR^m$. The fact that the latter has polynomial growth follows from the following bounds
\begin{equation*}
|\langle w,\widetilde{I}_{\psi_j\varphi}(\cdot,\eta)\rangle|=|\langle \mathcal{F}w,e^{i f(\cdot)\eta}\psi_j\varphi\rangle|\leq C_1\sup_{|\alpha|\leq k}\|\langle\cdot\rangle^k\partial^{\alpha}(e^{if(\cdot)\eta}\psi_j\varphi)\|_{L^{\infty}(\RR^m)}\leq C_2\langle\eta\rangle^k.
\end{equation*}
The fact that $I_{2;j}\in\SSS(\RR^m)$ follows from \eqref{est-cla-forskls2}. To show the very last equality, notice first that it trivially holds when $w\in\SSS(\RR^m)$. When $w\in\SSS'(\RR^m)$, pick $w_l\in\SSS(\RR^m)$, $l\in\ZZ_+$, such that $w_l\rightarrow w$ in $\SSS'(\RR^m)$. The set $\{\mathcal{F}w_l\}_{l\in\ZZ_+}$ is an equicontinuous subset of $\SSS'(\RR^m)$ since it is bounded in $\SSS'(\RR^m)$ and $\SSS(\RR^m)$ is barrelled. Hence, there are $C'_0>0$ and $k_0\in\ZZ_+$ such that
\begin{equation*}
|\langle w_l,\widetilde{I}_{\psi_j\varphi}(\cdot,\eta)\rangle|=|\langle \mathcal{F}w_l,e^{i f(\cdot)\eta}\psi_j\varphi\rangle|\leq C'_0\sup_{|\alpha|\leq k_0}\|\langle\cdot\rangle^{k_0}\partial^{\alpha}(e^{if(\cdot)\eta}\psi_j\varphi)\|_{L^{\infty}(\RR^m)}\leq C'_0C''\langle\eta\rangle^{k_0}
\end{equation*}
for all $l\in\ZZ_+$. Hence, dominated convergence implies
\begin{equation*}
\lim_{l\rightarrow\infty}\int_{\RR^m}\widetilde{\chi}_j(\eta)\mathcal{F}(\phi_j u)(\eta)\langle w_l,\widetilde{I}_{\psi_j\varphi}(\cdot,\eta)\rangle d\eta=\int_{\RR^m}\widetilde{\chi}_j(\eta)\mathcal{F}(\phi_j u)(\eta)\langle w,\widetilde{I}_{\psi_j\varphi}(\cdot,\eta)\rangle d\eta.
\end{equation*}
As $\langle w_l,I_{2;j}\rangle\rightarrow\langle w,I_{2;j}\rangle$ as $l\rightarrow\infty$, this completes the proof of the claim.
\end{proof}

\begin{remark}\label{rem-inv-dkslts}
Let the space $E$ be as in the theorem. If $f:O\rightarrow U$ is a diffeomorphism and $u\in\DD'(U)$, we can apply the theorem with $L=WF^E(u)$ to verify that $WF^E(f^*u)=f^*WF^E(u)$.
\end{remark}

One can employ the diffeomorphism invariance from Remark \ref{rem-inv-dkslts} and Theorem \ref{pull-back-diff} to show that $WF^E(u)$ and the space $\DD'^E_L$ can be invariantly defined on smooth manifolds, however we will not pursue this line of enquiry here.\\
\indent As we pointed out above, the theorem is applicable when $E=W^{r,p}(\RR^m)$, $1<p<\infty$, $r\in\RR$. In this case, we simply write $\DD'^{r,p}_L(O)$ for $\DD'^E_L(O)$ and we will denote the seminorms $\mathfrak{p}_{E;\varphi,\psi}$ by $\mathfrak{p}_{r,p;\varphi,\psi}$; notice that $\mathfrak{p}_{r,p;\varphi,\psi}(u)=\|\langle\cdot\rangle^r\psi\mathcal{F}(\varphi u)\|_{\mathcal{F}L^p(\RR^m)}$, $u\in\DD'^{r,p}_L(O)$. In the spacial case $L=\emptyset$, $\DD'^{r,p}_{\emptyset}(O)=W^{r,p}_{\operatorname{loc}}(O)$ topologically.\\
\indent We are now ready to state and prove the main result of the article: it generalises the constant rank case in \cite[Theorem 3.21]{PP} from $L^2$-Sobolev wave fronts to $L^p$-Sobolev wave fronts.

\begin{theorem}\label{the-for-pulbknosklskhtrk}
Let $O$ and $U$ be open sets in $\RR^m$ and $\RR^n$ respectively, let $f:O\rightarrow U$ be a smooth map of constant rank $k\in\ZZ_+$ and let $L$ be a closed conic subset of $U\times (\RR^n\backslash\{0\})$ satisfying $L\cap\mathcal{N}_f=\emptyset$. For each $p\in(1,\infty)$, the pullback $f^*:\mathcal{C}^{\infty}(U)\rightarrow \mathcal{C}^{\infty}(O)$, $f^*(u)=u\circ f$, uniquely extends to a well-defined and continuous mapping $f^*:\DD'^{r_2,p}_L(U)\rightarrow \DD'^{r_1,p}_{f^*L}(O)$ whenever $r_2-r_1>(n-k)/p$ and $r_2>(n-k)/p$. When $f$ is a submersion, this is valid for all $r_2\geq r_1$.
\end{theorem}

\begin{remark}
When $L=\emptyset$, the theorem states that $f^*:W^{r_2,p}_{\operatorname{loc}}(U)\rightarrow W^{r_1,p}_{\operatorname{loc}}(O)$ is well-defined and continuous whenever $r_1$ and $r_2$ satisfy the above conditions. Hence, if $m<n$, $O=U\cap \RR^m$ and $f$ is the canonical imbedding, the theorem can be viewed as a local variant of the Sobolev imbedding theorem for restricting to lower dimensional hyperplanes \cite[Theorem 4.12, p. 85]{adams}.
\end{remark}

\begin{proof}
Throughout the proof, $q$ will stand for the H\"older conjugate index to $p$. For $x\in\RR^m$, we denote $x=(x',x'')$ with $x'\in\RR^k$ and $x''\in\RR^{m-k}$. Furthermore, we denote by $0_l$ the zero in $\RR^l$, $l\in\ZZ_+$. Employing Theorem \ref{pull-back-diff} together with the constant rank theorem \cite[Theorem 4.12, p. 81]{lee} and arguing as in the proof \cite[Theorem 3.21, CASE 3]{PP}, one shows that it suffices to prove the following special case of the theorem.
\begin{itemize}
\item[$(*)$]Let $\widetilde{O}_0$ and $\widetilde{U}_0$ be open neighbourhoods of the origins in $\RR^m$ and $\RR^n$ respectively and let $\hat{f}_0:\widetilde{O}_0\rightarrow \widetilde{U}_0$, $\hat{f}_0(x)=(x',0_{n-k})$, $x=(x',x'')\in \widetilde{O}_0$. If $\widetilde{L}$ is a closed conic subset of $\widetilde{U}_0\times(\RR^n\backslash\{0\})$ which satisfies $\widetilde{L}\cap \mathcal{N}_{\hat{f}_0}=\emptyset$, then the map $\hat{f}_0^*:\mathcal{C}^{\infty}(\widetilde{U}_0)\rightarrow\mathcal{C}^{\infty}(\widetilde{O}_0)$ is continuous when $\mathcal{C}^{\infty}(\widetilde{U}_0)$ and $\mathcal{C}^{\infty}(\widetilde{O}_0)$ are equipped with the topologies induced by $\DD'^{r_2,p}_{\widetilde{L}}(\widetilde{U}_0)$ and $\DD'^{r_1,p}_{\hat{f}^*_0\widetilde{L}}(\widetilde{O}_0)$ respectively.
\end{itemize}
We prove $(*)$ by considering three cases.\\
\indent \underline{The case when $k=n$.} We can assume that $m>n$ since the claim is trivial if $m=n$. Notice that $\mathcal{N}_{\hat{f}_0}=\hat{f}_0(\widetilde{O}_0)\times \{0_n\}$ and
\begin{equation*}
\hat{f}_0^*\widetilde{L}=\{((x',x''),(\eta,0_{m-n}))\in \widetilde{O}_0\times(\RR^m\backslash \{0\})\,|\, (x',\eta)\in\widetilde{L}\}.\label{pul-bac-conic}
\end{equation*}
Let $\chi\in\XX_0(\RR^m)$ and $\varphi\in\DD(\widetilde{O}_0)\backslash\{0\}$ are such that $(\supp\varphi\times\supp\chi)\cap \hat{f}_0^*\widetilde{L}=\emptyset$; since our goal is to estimate $\mathfrak{p}_{r_1,p;\varphi,\chi}(\hat{f}^*_0u)$, we can assume that $\supp\chi\neq\emptyset$. Of course, $G:=\{0\}\cup\RR_+\supp\chi$ is a closed cone in $\RR^m$ which satisfies $(\supp\varphi\times G)\cap \hat{f}_0^*\widetilde{L}=\emptyset$. As in the proof \cite[Theorem 3.21, CASE 3]{PP}, we can find an open set $\widetilde{O}_1\subseteq \widetilde{O}_0$ containing $\supp\varphi$ and having a compact closure in $\widetilde{O}_0$ and open cones $G'_j$, $\widetilde{G}'_j$, $j=1,\ldots,s$, in $\RR^n$ and $0<\varepsilon<1/2$ which satisfy the following:
\begin{itemize}
\item[$(i)$] $G\backslash\{0_m\}\subseteq \widetilde{G}_0\cup\bigcup_{j=1}^s(G'_j\times\RR^{m-n})$ with $\widetilde{G}_0:=\{(\xi',\xi'')\in\RR^m\,|\, |\xi''|> \varepsilon|\xi'|\}$;
\item[$(ii)$] either all $G'_j$ and $\widetilde{G}'_j$, $j=1,\ldots,s$, are nonempty and they satisfy
\begin{equation*}
\overline{G'_j}\subseteq \widetilde{G}'_j\cup\{0_n\}\quad \mbox{and}\quad (\hat{f}_0(\overline{\widetilde{O}_1})\times \overline{\widetilde{G}'_j})\cap \widetilde{L}=\emptyset,\qquad j=1,\ldots s,
\end{equation*}
or $G'_j=\widetilde{G}'_j=\emptyset$ for all $j\in\{1,\ldots,s\}$.
\end{itemize}
In view of $(ii)$, a standard compactness argument yields that there is an open set $U_1\subseteq \widetilde{U}_0$ such that
\begin{itemize}
\item[$(iii)$] $\hat{f}_0(\overline{\widetilde{O}_1})\subseteq U_1$ and $(U_1\times \overline{\widetilde{G}'_j})\cap \widetilde{L}=\emptyset$, $j=1,\ldots,s$.
\end{itemize}
For each $j\in\{1,\ldots,s\}$, pick $\chi_j\in\XX_0(\RR^n)$ such that $0\leq \chi_j\leq1$, $\supp\chi_j\subseteq \widetilde{G}'_j\backslash B(0_n,1/3)$ and $\chi_j=1$ on $\overline{G'_j}\backslash B(0_n,1/2)$; when $G'_j$ and $\widetilde{G}'_j$ are empty, we take $\chi_j=0$. Furthermore, choose $\chi_0\in\XX_0(\RR^m)$ such that $0\leq \chi_0\leq 1$, $\supp\chi_0\subseteq \widetilde{G}'_0:=\{\xi\in\RR^m\,|\, |\xi''|>(\varepsilon/2)|\xi'|\}$ and $\chi_0=1$ on $\widetilde{G}_0\backslash B(0_m,1/2)$. Pick $\chi'\in\XX_0(\RR^m)$ such that $0\leq \chi'\leq 1$, $\chi'=0$ on $B(0_m,3/4)$ and $\chi'(x)=1$ when $|x|\geq 1$. Set
\begin{equation*}
\widetilde{\chi}_j:=\frac{(1-\chi_0)\chi'\chi(\chi_j\otimes\mathbf{1}_{\RR^{m-n}})}{\sum_{l=1}^s\chi_l\otimes \mathbf{1}_{\RR^{m-n}}},\quad j=1,\ldots,s;
\end{equation*}
when all $G'_j$ and $\widetilde{G}'_j$ are empty, we set $\widetilde{\chi}_j:=0$. In addition, denote $\widetilde{\chi}_0:=\chi_0\chi'\chi$. Then, $\widetilde{\chi}_j\in\XX_0(\RR^m)$,\footnote{This fact is not trivial for $\widetilde{\chi}_j$, $j=1,\ldots,s$, and it holds because of the properties of $\chi_0$; notice that if $\psi'\in\XX_0(\RR^n)$ then $\psi'\otimes\mathbf{1}_{\RR^{m-n}}$ may not be in $\XX(\RR^m)$!} $j=0,\ldots,s$, and $\chi'\chi=\sum_{j=0}^s\widetilde{\chi}_j$. Choose $\phi\in\DD(U_1)$ such that $\phi=1$ on a neighbourhood of $\hat{f}_0(\supp\varphi)$. Let $u\in\mathcal{C}^{\infty}(\widetilde{U}_0)$. We infer
\begin{equation}\label{est-for-uin-seminormsofcos}
\mathfrak{p}_{r_1,p;\varphi,\chi}(\hat{f}^*_0u)=\sup_{\theta\in\SSS(\RR^m),\,\ \|\theta\|_{\mathcal{F}^{-1}L^q(\RR^m)}\leq 1}|\langle \theta,\langle\cdot\rangle^{r_1}\chi\mathcal{F}(\varphi \hat{f}_0^*u)\rangle|.
\end{equation}
For $\theta\in\SSS(\RR^m)$ satisfying $\|\theta\|_{\mathcal{F}^{-1}L^q(\RR^m)}\leq 1$, we have
\begin{equation}\label{est-for-mai-reskls}
|\langle \theta,\langle\cdot\rangle^{r_1}\chi\mathcal{F}(\varphi \hat{f}_0^*u)\rangle|\leq |\langle \theta,\langle\cdot\rangle^{r_1}(1-\chi')\chi\mathcal{F}(\varphi \hat{f}_0^*u)\rangle|+\sum_{j=0}^s|\langle \theta,\langle\cdot\rangle^{r_1}\widetilde{\chi}_j\mathcal{F}(\varphi \hat{f}_0^*u)\rangle|.
\end{equation}
We first estimate the terms in the sum. For each $j\in\{0,\ldots,s\}$, it will be convenient to introduce the operator
\begin{equation*}
\mathcal{J}_j:\SSS(\RR^m)\rightarrow\SSS(\RR^n),\quad \mathcal{J}_j\vartheta(x'):=\int_{\RR^{m-n}}\varphi(x)\check{\widetilde{\chi}}_j(D)\langle D\rangle^{r_1}\mathcal{F}\vartheta(x)dx'';
\end{equation*}
clearly, $\mathcal{J}_j$ is well-defined and continuous. Since
\begin{equation*}
\mathcal{F}(\varphi \hat{f}_0^*u)(\xi)=\mathcal{F}(\varphi \hat{f}_0^*(\phi u))(\xi)=\frac{1}{(2\pi)^n}\int_{\RR^n}\mathcal{F}(\phi u)(\eta)\int_{\widetilde{O}_0}e^{ix'\eta-ix\xi}\varphi(x)dxd\eta,
\end{equation*}
we infer (we also employ the identity $\mathcal{F}(a\theta)=\check{a}(D)\mathcal{F}\theta$)
\begin{align}
\langle \theta,\langle\cdot\rangle^{r_1}\widetilde{\chi}_j\mathcal{F}(\varphi \hat{f}_0^*u)\rangle&=\frac{1}{(2\pi)^n}\iiint_{\RR^m\times\RR^n\times\RR^m} \theta(\xi)\mathcal{F}(\phi u)(\eta)e^{ix'\eta-ix\xi}\langle \xi\rangle^{r_1}\widetilde{\chi}_j(\xi)\varphi(x)d\xi d\eta dx\nonumber\\
&=\frac{1}{(2\pi)^n}\iint_{\RR^n\times\RR^m}\mathcal{F}(\phi u)(\eta)e^{ix'\eta}\varphi(x)\check{\widetilde{\chi}}_j(D)\langle D\rangle^{r_1}\mathcal{F}\theta(x)d\eta dx\label{exp-for-actonsplintskl}\\
&=\int_{\RR^n}\mathcal{F}(\phi u)(\eta)\mathcal{F}^{-1}(\mathcal{J}_j\theta)(\eta)d\eta.\label{exp-for-pai-ofdualityformainskl}
\end{align}
We first address the case when $j\in\{1,\ldots,s\}$; we only need to consider the case when $G'_j$ and $\widetilde{G}'_j$ are nonempty for otherwise $\mathcal{J}_j\theta=0$. Fix such $j$ and choose $\psi_j\in\XX_0(\RR^n)$ such that $0\leq \psi_j\leq 1$, $\supp\psi_j\subseteq \widetilde{G}'_j\backslash B(0_n,1/4)$ and $\psi_j=1$ on a neighbourhood of $\supp\chi_j$. Write $|\langle \theta,\langle\cdot\rangle^{r_1}\widetilde{\chi}_j\mathcal{F}(\varphi \hat{f}_0^*u)\rangle|\leq I_{1,j}+I_{2,j}$ with
\begin{align*}
I_{1,j}&:=\left|\int_{\RR^n}\psi_j(\eta)\mathcal{F}(\phi u)(\eta)\mathcal{F}^{-1}(\mathcal{J}_j\theta)(\eta)d\eta \right|,\\
I_{2,j}&:=\left|\int_{\RR^n}(1-\psi_j(\eta))\mathcal{F}(\phi u)(\eta)\mathcal{F}^{-1}(\mathcal{J}_j\theta)(\eta)d\eta\right|.
\end{align*}
We estimate $I_{1,j}$ as follows
\begin{align*}
I_{1,j}&\leq \|\langle\cdot\rangle^{r_2}\psi_j\mathcal{F}(\phi u)\|_{\mathcal{F}L^p(\RR^n)} \|\langle\cdot\rangle^{-r_2}\mathcal{F}^{-1}(\mathcal{J}_j\theta)\|_{\mathcal{F}^{-1}L^q(\RR^n)}\\
&=\mathfrak{p}_{r_2,p;\phi,\psi_j}(u)\|\langle D\rangle^{-r_2}\mathcal{J}_j\theta\|_{L^q(\RR^n)};
\end{align*}
notice that $\mathfrak{p}_{r_2,p;\phi,\psi_j}$ is a well-defined continuous seminorm on $\DD'^{r_2,p}_{\widetilde{L}}(\widetilde{U}_0)$ in view of $(iii)$. Employing the H\"older inequality, we infer
\begin{align}
\|&\langle D\rangle^{-r_2}\mathcal{J}_j\theta\|_{L^q(\RR^n)}\nonumber\\
&= \left(\int_{\RR^n}\left|\int_{\RR^{m-n}}\langle D_{x'}\rangle^{-r_2}\varphi(x)\check{\widetilde{\chi}}_j(D_x)\langle D_x\rangle^{r_1}\mathcal{F}\theta(x)dx''\right|^qdx'\right)^{1/q}\nonumber\\
&\leq\|\langle \cdot\rangle^{-m+n}\|_{L^p(\RR^{m-n})}\left(\int_{\RR^m}\left|\langle D_{x'}\rangle^{-r_2}\langle x''\rangle^{m-n} \varphi(x)\check{\widetilde{\chi}}_j(D_x)\langle D_x\rangle^{r_1}\mathcal{F}\theta(x)\right|^qdx\right)^{1/q}.\label{est-for-operatoronthetwithders}
\end{align}
We claim that the operator that acts on $\mathcal{F}\theta$ is continuous on $L^q(\RR^m)$. To see this, pick $\widetilde{\chi}\in\XX_0(\RR^m)$ such that $0\leq\widetilde{\chi}\leq 1$, $\supp\widetilde{\chi}\subseteq \{\xi\in\RR^m\,|\, |\xi'|>|\xi''|/(4\varepsilon)\}\backslash B(0_m,1/4)$ and $\widetilde{\chi}=1$ on $\{\xi\in\RR^m\,|\, |\xi'|>|\xi''|/(2\varepsilon)\}\backslash B(0_m,1/2)$ and write it as
\begin{align*}
\langle& D_{x'}\rangle^{-r_2}\langle x''\rangle^{m-n}\varphi(x)\check{\widetilde{\chi}}_j(D_x)\langle D_x\rangle^{r_1}\\
&=\left(\langle D_{x'}\rangle^{-r_2}\widetilde{\chi}(D_x)\langle D_x\rangle^{r_1}\right)\left(\langle D_x\rangle^{-r_1}\langle x''\rangle^{m-n}\varphi(x)\check{\widetilde{\chi}}_j(D_x)\langle D_x\rangle^{r_1}\right)\\
&{}\quad+\langle D_{x'}\rangle^{-r_2}(\operatorname{Id}-\widetilde{\chi}(D_x))\langle x''\rangle^{m-n}\varphi(x)\check{\widetilde{\chi}}_j(D_x)\langle D_x\rangle^{r_1}.
\end{align*}
The operator $\langle D_x\rangle^{-r_1}\langle x''\rangle^{m-n}\varphi(x)\check{\widetilde{\chi}}_j(D_x)\langle D_x\rangle^{r_1}$ is a $\Psi$DO with symbol in $S^0_{1,0}(\RR^{2m})$ and hence continuous on $L^q(\RR^m)$. The Fourier multiplier $\langle D_{x'}\rangle^{-r_2}\widetilde{\chi}(D_x)\langle D_x\rangle^{r_1}$ has symbol $\langle \xi'\rangle^{-r_2}\widetilde{\chi}(\xi)\langle\xi\rangle^{r_1}$, which, when viewed as a $\Psi$DO, also has a symbol in $S^0_{1,0}(\RR^{2m})$ since $r_2\geq r_1$ and $|\xi'|>|\xi''|/(4\varepsilon)$ on the support of $\widetilde{\chi}$. The second operator is a $\Psi$DO with symbol in $S^{-\infty}(\RR^{2m})$. This follows from the fact that $(\operatorname{Id}-\widetilde{\chi}(D_x))\langle x''\rangle^{m-n}\varphi(x)\check{\widetilde{\chi}}_j(D_x)\langle D_x\rangle^{r_1}$ is a $\Psi$DO with symbol in $S^{-\infty}(\RR^{2m})$ which, in turn, follows by applying the asymptotic expansion for the composition and the fact that for any $N\in\ZZ_+$
\begin{equation*}
\sum_{|\alpha|=0}^{N-1}\alpha!^{-1}\partial^{\alpha}_{\xi}(1-\widetilde{\chi}(\xi)) D^{\alpha}_x\left(\langle x''\rangle^{m-n}\varphi(x)\check{\widetilde{\chi}}_j(\xi)\langle \xi\rangle^{r_1}\right)=0,\quad (x,\xi)\in\RR^{2m},
\end{equation*}
in view of the properties of $\widetilde{\chi}$ and $\widetilde{\chi}_j$. Consequently, we deduce that
\begin{equation*}
I_{1,j}\leq C'_1\mathfrak{p}_{r_2,p;\phi,\psi_j}(u)\|\mathcal{F}\theta\|_{L^q(\RR^m)}\leq C'_1\mathfrak{p}_{r_2,p;\phi,\psi_j}(u).
\end{equation*}
We claim that $I_{2,j}$ is bounded by a continuous seminorm on $\DD'(\widetilde{U}_0)$ of $u$. To prove this, it suffices to show that
\begin{equation*}
\DD'(\widetilde{U}_0)\rightarrow[0,\infty),\quad v\mapsto \sup_{\theta\in\SSS(\RR^m),\, \|\theta\|_{\mathcal{F}^{-1}L^q(\RR^m)}\leq 1} |\langle v, \phi(\operatorname{Id}-\check{\psi}_j(D))\mathcal{J}_j\theta\rangle|,
\end{equation*}
is a well-defined continuous seminorm on $\DD'(\widetilde{U}_0)$. To prove the latter, it is enough to show that $\{(\operatorname{Id}-\check{\psi}_j(D))\mathcal{J}_j\theta\,|\, \theta\in\SSS(\RR^m),\, \|\mathcal{F}\theta\|_{L^q(\RR^m)}\leq 1\}$ is a bounded subset of $\mathcal{C}^{\infty}(\RR^n)$. For $\alpha\in\NN^n$, similarly as in \eqref{est-for-operatoronthetwithders}, we have
\begin{multline*}
\|\partial^{\alpha}(\operatorname{Id}-\check{\psi}_j(D))\mathcal{J}_j\theta\|_{L^q(\RR^n)}\leq\|\langle \cdot\rangle^{-m+n}\|_{L^p(\RR^{m-n})}\\
\cdot\left(\int_{\RR^m}\left|\partial^{\alpha}_{x'}\left((\operatorname{Id}-\check{\psi}_j(D_{x'}))\langle x''\rangle^{m-n}\varphi(x)\check{\widetilde{\chi}}_j(D_x)\langle D_x\rangle^{r_1}\mathcal{F}\theta(x)\right)\right|^qdx\right)^{1/q}.
\end{multline*}
We claim that the operator that acts on $\mathcal{F}\theta$ has symbol in $S^{-\infty}(\RR^{2m})$ and thus it is continuous on $L^q(\RR^m)$. To see this, let $\widetilde{\chi}$ be as above and write
\begin{align*}
(\operatorname{Id}-\check{\psi}_j(D_{x'}))&\langle x''\rangle^{m-n}\varphi(x)\check{\widetilde{\chi}}_j(D_x)\langle D_x\rangle^{r_1}\\
&=(\operatorname{Id}-\check{\psi}_j(D_{x'}))\widetilde{\chi}(D_x)\langle x''\rangle^{m-n}\varphi(x)\check{\widetilde{\chi}}_j(D_x)\langle D_x\rangle^{r_1}\\
&{}\quad +(\operatorname{Id}-\check{\psi}_j(D_{x'}))(\operatorname{Id}-\widetilde{\chi}(D_x))\langle x''\rangle^{m-n}\varphi(x)\check{\widetilde{\chi}}_j(D_x)\langle D_x\rangle^{r_1}.
\end{align*}
Arguing as above, one can show that the second operator has symbol in $S^{-\infty}(\RR^{2m})$. To show this for the first, notice that the symbol of $(\operatorname{Id}-\check{\psi}_j(D_{x'}))\widetilde{\chi}(D_x)$ is in $S^0_{1,0}(\RR^{2m})$ in view of the support of $\widetilde{\chi}$. The asymptotic expansion of the composition and the fact that for any $N\in\ZZ_+$
\begin{equation*}
\sum_{|\beta|=0}^{N-1}\beta!^{-1}\partial^{\beta}_{\xi}\left((1-\check{\psi}_j(\xi'))\widetilde{\chi}(\xi)\right) D^{\beta}_x\left(\langle x''\rangle^{m-n}\varphi(x)\check{\widetilde{\chi}}_j(\xi)\langle \xi\rangle^{r_1}\right)=0,\quad (x,\xi)\in\RR^{2m},
\end{equation*}
which holds in view of the properties of $\psi_j$ and $\widetilde{\chi}_j$, imply that the symbol of the first operator is also in $S^{-\infty}(\RR^{2m})$. We deduce that $\{\partial^{\alpha}(\operatorname{Id}-\check{\psi}_j(D))\mathcal{J}_j\theta\,|\, \theta\in\SSS(\RR^m),\, \|\mathcal{F}\theta\|_{L^q(\RR^m)}\leq1\}$ is bounded in $L^q(\RR^n)$ for all $\alpha\in\NN^n$ and thus $\{(\operatorname{Id}-\check{\psi}_j(D))\mathcal{J}_j\theta\,|\, \theta\in\SSS(\RR^m),\, \|\mathcal{F}\theta\|_{L^q(\RR^m)}\leq1\}$ is a bounded subsets of $\DD_{L^q}(\RR^n)$. Since $\DD_{L^q}(\RR^n)\subseteq\DD_{L^{\infty}}(\RR^n)$ continuously, the set is bounded in $\DD_{L^{\infty}}(\RR^n)$ and consequently in $\mathcal{C}^{\infty}(\RR^n)$ as well. This completes the proof for the terms when $j\in\{1,\ldots,s\}$.\\
\indent We claim that the term when $j=0$ in the sum in \eqref{est-for-mai-reskls} is bounded by a continuous seminorm on $\DD'(\widetilde{U}_0)$ of $u$. In view of \eqref{exp-for-pai-ofdualityformainskl}, this immediately follows once we show that
\begin{equation*}
\DD'(\widetilde{U}_0)\rightarrow[0,\infty),\quad v\mapsto \sup_{\theta\in\SSS(\RR^m),\, \|\theta\|_{\mathcal{F}^{-1}L^q(\RR^m)}\leq 1} |\langle v, \phi\mathcal{J}_0\theta\rangle|,
\end{equation*}
is a well-defined continuous seminorm on $\DD'(\widetilde{U}_0)$. To prove the latter, it suffices to show that $\{\mathcal{J}_0\theta\,|\, \theta\in\SSS(\RR^m),\, \|\mathcal{F}\theta\|_{L^q(\RR^m)}\leq 1\}$ is a bounded subset of $\mathcal{C}^{\infty}(\RR^n)$. Fix $\alpha\in\NN^n$ and write
\begin{equation*}
\partial^{\alpha}\mathcal{J}_0\theta(x')=\int_{\RR^{m-n}}\partial^{\alpha}_{x'}\left(\left(\langle D_{x''}\rangle^{|\alpha|+r_1}\varphi(x)\right)\langle D_{x''}\rangle^{-|\alpha|-r_1} \langle D_x\rangle^{r_1}\check{\widetilde{\chi}}_0(D_x)\mathcal{F}\theta(x)\right)dx''.
\end{equation*}
The H\"older inequality gives
\begin{multline*}
\|\partial^{\alpha}\mathcal{J}_0\theta\|_{L^q(\RR^n)}\leq \|\langle\cdot\rangle^{-m+n}\|_{L^p(\RR^{m-n})}\\
\cdot\left(\int_{\RR^m}\left|\partial^{\alpha}_{x'}\left(\langle x''\rangle^{m-n}\left(\langle D_{x''}\rangle^{|\alpha|+r_1}\varphi\right)(x) \langle D_{x''}\rangle^{-|\alpha|-r_1} \langle D_x\rangle^{r_1}\check{\widetilde{\chi}}_0(D_x) \mathcal{F}\theta(x)\right)\right|^qdx\right)^{1/q}.
\end{multline*}
Notice that the $\Psi$DO $\langle x''\rangle^{m-n}\left(\langle D_{x''}\rangle^{|\alpha|+r_1}\varphi\right)(x)\langle D_{x''}\rangle^{-|\alpha|-r_1} \langle D_x\rangle^{r_1}\check{\widetilde{\chi}}_0(D_x)$ has a symbol in $S^{-|\alpha|}_{1,0}(\RR^{2m})$ in view of the support of $\widetilde{\chi}_0$. Consequently, its composition with $\partial^{\alpha}_{x'}$ is a $\Psi$DO with symbol in $S^0_{1,0}(\RR^{2m})$ and hence it is continuous on $L^q(\RR^m)$. We deduce that $\{\partial^{\alpha}\mathcal{J}_0\theta\,|\, \theta\in\SSS(\RR^m),\, \|\mathcal{F}\theta\|_{L^q(\RR^m)}\leq 1\}$ is bounded in $L^q(\RR^n)$ for all $\alpha\in\NN^n$. Arguing as above, this implies that $\{\mathcal{J}_0\theta\,|\, \theta\in\SSS(\RR^m),\, \|\mathcal{F}\theta\|_{L^q(\RR^m)}\leq 1\}$ is bounded in $\mathcal{C}^{\infty}(\RR^n)$ which completes the proof for the term when $j=0$. The first term on the right hand side in \eqref{est-for-mai-reskls} is also bounded by a continuous seminorm on $\DD'(\widetilde{U}_0)$ of $u$. This can be shown similarly as for the term with $j=0$ by rewriting it similarly as in \eqref{exp-for-pai-ofdualityformainskl} and employing the fact that $(1-\chi')\chi$ has compact support. We showed that $\mathfrak{p}_{r_1,p;\varphi,\chi}(\hat{f}^*_0u)$ is bounded by a continuous seminorm on $\DD'^{r_2,p}_{\widetilde{L}}(\widetilde{U}_0)$ of $u$.\\
\indent It remains to bound $\mathfrak{p}(\hat{f}^*_0u)$ with $\mathfrak{p}$ a continuous seminorm on $\DD'(\widetilde{O}_0)$. Without loss in generality, we can assume that $\mathfrak{p}(\hat{f}^*_0u)=\sup_{\chi\in B}|\langle \hat{f}^*_0u,\chi\rangle|$ where $B$ is a bounded subset of $\DD(\widetilde{O}_0)$. There is a compact set $K\subseteq\widetilde{O}_0$ such that $B$ is a bounded subset of $\DD_K(\RR^m)$. Pick $\phi\in\DD(\widetilde{U}_0)$ such that $\phi=1$ on a neighbourhood of $\hat{f}_0(K)$. Similarly as above, denote $\mathcal{J}:\SSS(\RR^m)\rightarrow\SSS(\RR^n)$, $\mathcal{J}\vartheta(x'):=\int_{\RR^{m-n}}\vartheta(x)dx''$; clearly $\mathcal{J}$ is well-defined and continuous. Notice that
\begin{equation*}
\sup_{\chi\in B}|\langle\hat{f}^*_0u,\chi\rangle|=\sup_{\chi\in B}|\langle\hat{f}^*_0(\phi u),\chi\rangle|=\sup_{\chi\in B}\left|\int_{\RR^m}\phi(x') u(x') \chi(x) dx\right|=\sup_{\chi\in B}|\langle u,\phi\mathcal{J}\chi\rangle|.
\end{equation*}
Since $\{\phi\mathcal{J}\chi\,|\, \chi \in B\}$ is a bounded subset of $\DD(\widetilde{U}_0)$, the right hand side is a continuous seminorm on $\DD'(\widetilde{U}_0)$ of $u$ and the proof is complete.\\
\indent \underline{The case when $k=m$.} We can assume $m<n$ since the case $m=n$ is trivial. For $y\in\RR^n$, we write $y=(y',y'')$ with $y'\in\RR^m$ and $y''\in\RR^{n-m}$. We have
\begin{align}
\mathcal{N}_{\hat{f}_0}&=\{((x,0_{n-m}),(0_m,\eta''))\in \widetilde{U}_0\times\RR^n\,|\, x\in\widetilde{O}_0\},\label{nor-for-map-injec111111}\\
\hat{f}_0^*\widetilde{L}&=\{(x,\eta')\in \widetilde{O}_0\times(\RR^m\backslash \{0\})\,|\, \exists\eta''\in\RR^{n-m},\, ((x,0_{n-m}),(\eta',\eta''))\in\widetilde{L}\}.\nonumber
\end{align}
Let $\chi\in\XX_0(\RR^m)$ and $\varphi\in\DD(\widetilde{O}_0)\backslash\{0\}$ be such that $(\supp\varphi\times\supp\chi)\cap \hat{f}_0^*\widetilde{L}=\emptyset$; since our goal is to estimate $\mathfrak{p}_{r_1,p;\varphi,\chi}(\hat{f}^*_0u)$, we can assume that $\supp\chi\neq\emptyset$. Then $G:=\{0\}\cup\RR_+\supp\chi$ is a closed cone in $\RR^m$ which satisfies $(\supp\varphi\times G)\cap \hat{f}_0^*\widetilde{L}=\emptyset$. As in the proof of \cite[Theorem 3.21, CASE 3]{PP}, we can find an open set $\widetilde{O}_1\subseteq \widetilde{O}_0$ containing $\supp\varphi$ and having a compact closure in $\widetilde{O}_0$, an open set $\widetilde{U}_1\subseteq \widetilde{U}_0$ containing $\hat{f}_0(\supp\varphi)$, nonempty open cones $G'_j$, $\widetilde{G}'_j$, $j=1,\ldots,s$, in $\RR^m$ and $C_0>1$ which satisfy the following:
\begin{itemize}
\item[$(i)'$] $G\backslash\{0_m\}\subseteq \bigcup_{j=1}^sG'_j$;
\item[$(ii)'$] $\overline{G'_j}\subseteq \widetilde{G}'_j\cup\{0_m\}$ and $(\hat{f}_0(\overline{\widetilde{O}_1})\times \overline{\widetilde{G}'_j}\times\RR^{n-m})\cap \widetilde{L}=\emptyset$, $j=1,\ldots s$;
\item[$(iii)'$] $(\widetilde{U}_1\times \overline{\widetilde{V}_0})\cap \widetilde{L}=\emptyset$ with $\widetilde{V}_0:=\{(\eta',\eta'')\in\RR^n\,|\, |\eta''|>C_0|\eta'|\}$.
\end{itemize}
The second part of $(ii)'$ together with a standard compactness argument implies that there is an open set $U_1\subseteq \widetilde{U}_0$ such that
\begin{itemize}
\item[$(iv)'$] $\hat{f}_0(\overline{\widetilde{O}_1})\subseteq U_1$ and $(U_1\times \overline{\widetilde{G}'_j}\times \RR^{n-m})\cap \widetilde{L}=\emptyset$, $j=1,\ldots,s$.
\end{itemize}
For each $j\in\{1,\ldots,s\}$, pick $\chi_j\in\XX_0(\RR^m)$ such that $0\leq \chi_j\leq 1$, $\supp\chi_j\subseteq \widetilde{G}'_j\backslash B(0_m,1/3)$ and $\chi_j=1$ on $\overline{G'_j}\backslash B(0_m,1/2)$. Choose $\chi'\in\XX_0(\RR^m)$ such that $0\leq \chi'\leq 1$, $\chi'=0$ on $B(0_m,3/4)$ and $\chi'(x)=1$ when $|x|\geq 1$. Set
\begin{equation*}
\widetilde{\chi}_j:=\frac{\chi'\chi\chi_j}{\sum_{l=1}^s\chi_l},\quad j=1,\ldots,s.
\end{equation*}
Then, $\widetilde{\chi}_j\in\XX_0(\RR^m)$, $j=1,\ldots,s$, and $\chi'\chi=\sum_{j=1}^s\widetilde{\chi}_j$. Pick $\phi\in\DD(U_1\cap\widetilde{U}_1)$ such that $\phi=1$ on a neighbourhood of $\hat{f}_0(\supp\varphi)$. Given $u\in\mathcal{C}^{\infty}(\widetilde{U}_0)$, we again write \eqref{est-for-uin-seminormsofcos}. For $\theta\in\SSS(\RR^m)$ satisfying $\|\theta\|_{\mathcal{F}^{-1}L^q(\RR^m)}\leq 1$, we have
\begin{equation}\label{est-for-mai-reskls11}
|\langle \theta,\langle\cdot\rangle^{r_1}\chi\mathcal{F}(\varphi \hat{f}_0^*u)\rangle|\leq |\langle \theta,\langle\cdot\rangle^{r_1}(1-\chi')\chi\mathcal{F}(\varphi \hat{f}_0^*u)\rangle|+\sum_{j=1}^s|\langle \theta,\langle\cdot\rangle^{r_1}\widetilde{\chi}_j\mathcal{F}(\varphi \hat{f}_0^*u)\rangle|.
\end{equation}
We first estimate the terms in the sum. For $j\in\{1,\ldots,s\}$, pick $\psi_j\in\XX_0(\RR^m)$ such that $0\leq \psi_j\leq 1$, $\supp\psi_j\subseteq \widetilde{G}'_j\backslash B(0_m,1/4)$ and $\psi_j=1$ on a neighbourhood of $\supp\chi_j$. Set $\widetilde{\psi}_j:=\psi_j\otimes\mathbf{1}_{\RR^{n-m}}$. In addition, choose $\widetilde{\psi}\in\XX_0(\RR^n)$ such that $0\leq \widetilde{\psi}\leq 1$, $\supp\widetilde{\psi}\subseteq \widetilde{V}_0\backslash B(0_n,1/2)$ and $\widetilde{\psi}=1$ on $\{(\eta',\eta'')\in\RR^n\,|\, |\eta''|>2C_0|\eta'|\}\backslash B(0_n,1)$. Similarly as in \eqref{exp-for-actonsplintskl}, we infer
\begin{align*}
\langle \theta,\langle\cdot\rangle^{r_1}\widetilde{\chi}_j\mathcal{F}(\varphi \hat{f}_0^*u)\rangle& =\frac{1}{(2\pi)^n}\iint_{\RR^n\times\RR^m}\mathcal{F}(\phi u)(\eta)e^{ix\eta'}\varphi(x)\check{\widetilde{\chi}}_j(D)\langle D\rangle^{r_1}\mathcal{F}\theta(x)d\eta dx\\
&=I'_{1,j}+I'_{2,j}+I'_{3,j},
\end{align*}
with
\begin{align*}
I'_{1,j}&:=\frac{1}{(2\pi)^{n-m}}\int_{\RR^n}\mathcal{F}(\phi u)(\eta) \widetilde{\psi}_j(\eta)(1-\widetilde{\psi}(\eta))\mathcal{F}^{-1}(\varphi\check{\widetilde{\chi}}_j(D)\langle D\rangle^{r_1}\mathcal{F}\theta)(\eta')d\eta,\\
I'_{2,j}&:=\frac{1}{(2\pi)^{n-m}}\int_{\RR^n}\mathcal{F}(\phi u)(\eta)\widetilde{\psi}(\eta)\mathcal{F}^{-1}(\varphi\check{\widetilde{\chi}}_j(D)\langle D\rangle^{r_1}\mathcal{F}\theta)(\eta')d\eta,\\
I'_{3,j}&:=\frac{1}{(2\pi)^{n-m}}\int_{\RR^n}\mathcal{F}(\phi u)(\eta) (1-\widetilde{\psi}_j(\eta))(1-\widetilde{\psi}(\eta))\mathcal{F}^{-1}(\varphi\check{\widetilde{\chi}}_j(D)\langle D\rangle^{r_1}\mathcal{F}\theta)(\eta')d\eta.
\end{align*}
We need the following fact.\\
\\
\noindent \textbf{Claim} Let $\delta_d$ be the delta distribution on $\RR^d$. Then $\langle D\rangle^{-\lambda}\delta_d\in L^q(\RR^d)$, $\lambda>d/p$.\\
\\
We defer its proof for later and continue with the proof of the theorem. We first estimate $I'_{1,j}$. Pick $\lambda>(n-m)/p$ such that $r_2-r_1\geq\lambda$ and $r_2\geq\lambda$. The Claim implies that $\langle D_{y''}\rangle^{-\lambda}\delta_{n-m}\in L^q(\RR^{n-m})$. Notice that
\begin{align}
|I'_{1,j}|&\leq (2\pi)^{m-n}\|\langle\cdot\rangle^{r_2}\widetilde{\psi}_j(1-\widetilde{\psi})\mathcal{F}(\phi u)\|_{\mathcal{F}L^p(\RR^n)}\nonumber\\
&\quad\cdot\|\langle\cdot\rangle^{-r_2}(\mathcal{F}^{-1}(\varphi\check{\widetilde{\chi}}_j(D)\langle D\rangle^{r_1}\mathcal{F}\theta)\otimes\mathbf{1}_{\RR^{n-m}})\|_{\mathcal{F}^{-1}L^q(\RR^n)}\nonumber\\
&= \|\langle\cdot\rangle^{r_2}\widetilde{\psi}_j(1-\widetilde{\psi})\mathcal{F}(\phi u)\|_{\mathcal{F}L^p(\RR^n)}\nonumber\\
&\quad\cdot\|\langle D_y\rangle^{-r_2}(\varphi\otimes\mathbf{1}_{\RR^{n-m}})\check{\widetilde{\chi}}_j(D_{y'})\langle D_{y'}\rangle^{r_1}\langle D_{y''}\rangle^{\lambda}(\mathcal{F}\theta\otimes\langle D_{y''}\rangle^{-\lambda}\delta_{n-m})\|_{L^q(\RR^n)}.\label{ine-for-iprjklskr}
\end{align}
We claim that the operator acting on $\mathcal{F}\theta\otimes\langle D_{y''}\rangle^{-\lambda}\delta_{n-m}$ is continuous on $L^q(\RR^n)$. To see this, write it as
\begin{multline*}
\langle D_y\rangle^{-r_2}(\varphi\otimes\mathbf{1}_{\RR^{n-m}})\check{\widetilde{\chi}}_j(D_{y'})\langle D_{y'}\rangle^{r_1}\langle D_{y''}\rangle^{\lambda}\\
=\left(\langle D_y\rangle^{-r_2}(\varphi\otimes\mathbf{1}_{\RR^{n-m}})\langle D_y\rangle^{r_2}\right)\left(\langle D_y\rangle^{-r_2}\check{\widetilde{\chi}}_j(D_{y'})\langle D_{y''}\rangle^{\lambda}\langle D_{y'}\rangle^{r_1}\right).
\end{multline*}
The $\Psi$DO $\langle D_y\rangle^{-r_2}(\varphi\otimes\mathbf{1}_{\RR^{n-m}})\langle D_y\rangle^{r_2}$ has symbol in $S^0_{1,0}(\RR^{2n})$ and hence it is continuous on $L^q(\RR^n)$. The second operator is a Fourier multiplier with symbol $\langle \eta\rangle^{-r_2}\check{\widetilde{\chi}}_j(\eta')\langle \eta''\rangle^{\lambda}\langle \eta'\rangle^{r_1}$. When $r_1\geq 0$, since $r_2\geq r_1+\lambda$, it is straightforward to verify that it satisfies the conditions of the Lizorkin-Marcinkiewicz multiplier theorem \cite[Corollary 6.2.5, p. 444]{Grafakos} (cf. \cite[Theorem 6.2.4, p. 441]{Grafakos}) and hence it is continuous on $L^q(\RR^n)$. When $r_1<0$ then one can check that both $\langle \eta'\rangle^{r_1}$ and $\langle \eta\rangle^{-r_2}\check{\widetilde{\chi}}_j(\eta')\langle \eta''\rangle^{\lambda}$ satisfy the conditions of the Lizorkin-Marcinkiewicz multiplier theorem (since $r_2\geq \lambda$) and thus both are continuous on $L^q(\RR^n)$. We deduce
\begin{equation*}
|I'_{1,j}|\leq C'_2\|\langle\cdot\rangle^{r_2}\widetilde{\psi}_j(1-\widetilde{\psi})\mathcal{F}(\phi u)\|_{\mathcal{F}L^p(\RR^n)}\|\mathcal{F}\theta\|_{L^q(\RR^m)}\|\langle D\rangle^{-\lambda}\delta_{n-m}\|_{L^q(\RR^{n-m})}\leq C'_3\mathfrak{p}_{r_2,p;\phi,\widetilde{\psi}_j(1-\widetilde{\psi})}(u)
\end{equation*}
and $\mathfrak{p}_{r_2,p;\phi,\widetilde{\psi}_j(1-\widetilde{\psi})}$ is indeed a well-defined continuous seminorm on $\DD'^{r_2,p}_{\widetilde{L}}(\widetilde{U}_0)$ in view of $(iv)'$ and the fact that $\widetilde{\psi}_j(1-\widetilde{\psi})\in\XX_0(\RR^n)$ (this holds in view of the properties of $\widetilde{\psi}$). With $\lambda$ as above, one analogously obtains
\begin{align*}
|I'_{2,j}|&\leq \|\langle\cdot\rangle^{r_2}\widetilde{\psi}\mathcal{F}(\phi u)\|_{\mathcal{F}L^p(\RR^n)}\\
&{}\quad\cdot\|\langle D_y\rangle^{-r_2}(\varphi\otimes\mathbf{1}_{\RR^{n-m}})\check{\widetilde{\chi}}_j(D_{y'})\langle D_{y'}\rangle^{r_1}\langle D_{y''}\rangle^{\lambda}(\mathcal{F}\theta\otimes\langle D_{y''}\rangle^{-\lambda}\delta_{n-m})\|_{L^q(\RR^n)}\\
&\leq C'_3\mathfrak{p}_{r_2,p;\phi,\widetilde{\psi}}(u)
\end{align*}
and $\mathfrak{p}_{r_2,p;\phi,\widetilde{\psi}}$ is well-defined continuous seminorm on $\DD'^{r_2,p}_{\widetilde{L}}(\widetilde{U}_0)$ in view of $(iii)'$. Finally, we show that $I'_{3,j}$ is bounded by a continuous seminorm on $\DD'(\widetilde{U}_0)$ of $u$. Fix $\lambda>(n-m)/p$. It suffices to prove that
\begin{multline*}
\DD'(\widetilde{U}_0)\rightarrow[0,\infty),\,\, v\mapsto \sup_{\theta\in\SSS(\RR^m),\|\theta\|_{\mathcal{F}^{-1}L^q(\RR^m)}\leq 1} |\langle v, \phi(\operatorname{Id}-\check{\widetilde{\psi}}_j(D_y))(\operatorname{Id}-\check{\widetilde{\psi}}(D_y))(\varphi\otimes\mathbf{1}_{\RR^{n-m}})\\
\check{\widetilde{\chi}}_j(D_{y'})\langle D_{y'}\rangle^{r_1}\langle D_{y''}\rangle^{\lambda}(\mathcal{F}\theta\otimes\langle D_{y''}\rangle^{-\lambda}\delta_{n-m})\rangle|,
\end{multline*}
is a well-defined continuous seminorm on $\DD'(\widetilde{U}_0)$. As before, it is enough to show that $\{(\operatorname{Id}-\check{\widetilde{\psi}}_j(D_y))(\operatorname{Id}-\check{\widetilde{\psi}}(D_y))(\varphi\otimes\mathbf{1}_{\RR^{n-m}})
\check{\widetilde{\chi}}_j(D_{y'})\langle D_{y'}\rangle^{r_1}\langle D_{y''}\rangle^{\lambda}(\mathcal{F}\theta\otimes\langle D_{y''}\rangle^{-\lambda}\delta_{n-m})\,|\, \theta\in\SSS(\RR^m),\, \|\mathcal{F}\theta\|_{L^q(\RR^m)}\leq 1\}$ is a bounded subset of $\DD_{L^q}(\RR^n)$. We are going to show that $(\operatorname{Id}-\check{\widetilde{\psi}}_j(D_y))(\operatorname{Id}-\check{\widetilde{\psi}}(D_y))(\varphi\otimes\mathbf{1}_{\RR^{n-m}})
\check{\widetilde{\chi}}_j(D_{y'})\langle D_{y'}\rangle^{r_1}\langle D_{y''}\rangle^{\lambda}$ is a $\Psi$DO with symbol in $S^{-\infty}(\RR^{2n})$ which would immediately imply this. The properties of $\widetilde{\psi}$ yield that $(\operatorname{Id}-\check{\widetilde{\psi}}_j(D_y))(\operatorname{Id}-\check{\widetilde{\psi}}(D_y))(\varphi\otimes\mathbf{1}_{\RR^{n-m}})$ is a $\Psi$DO with symbol in $S^0_{1,0}(\RR^{2n})$ (in fact, $(1-\widetilde{\psi}_j)(1-\widetilde{\psi})\in\XX(\RR^n)$). Hence, for $N\in\ZZ_+$, the asymptotic expansion of the composition gives
\begin{multline*}
(\operatorname{Id}-\check{\widetilde{\psi}}_j(D_y))(\operatorname{Id}-\check{\widetilde{\psi}}(D_y))(\varphi\otimes\mathbf{1}_{\RR^{n-m}})
\check{\widetilde{\chi}}_j(D_{y'})\langle D_{y'}\rangle^{r_1}\langle D_{y''}\rangle^{\lambda}\\
=\left(\sum_{|\beta|=0}^{N-1}\beta!^{-1}D^{\beta}_y\varphi(y') \partial^{\beta}_{\eta}\left((1-\check{\widetilde{\psi}}_j)(1-\check{\widetilde{\psi}})\right)(D_y)+b_N(y,D_y)\right) \check{\widetilde{\chi}}_j(D_{y'})\langle D_{y'}\rangle^{r_1}\langle D_{y''}\rangle^{\lambda},
\end{multline*}
with $b_N\in S^{-N}_{1,0}(\RR^{2n})$. Since $\supp(1-\check{\widetilde{\psi}}_j)\cap\supp(\check{\widetilde{\chi}}_j\otimes\mathbf{1}_{\RR^{n-m}})=\emptyset$, the above equals $b_N(y,D_y)\check{\widetilde{\chi}}_j(D_{y'})\langle D_{y'}\rangle^{r_1}\langle D_{y''}\rangle^{\lambda}$ which is a $\Psi$DO with symbol in $S^{-N+\max\{r_1,0\}+\lambda}_{0,0}(\RR^{2n})$. As $N$ was arbitrary, we deduce the desired claim.\\
\indent We now address the first term on the right in \eqref{est-for-mai-reskls11}. Notice that
\begin{multline}\label{equ-for-semofpulbkhsfs}
\sup_{\theta\in\SSS(\RR^m),\, \|\theta\|_{\mathcal{F}^{-1}L^q(\RR^m)}\leq1}|\langle \theta,\langle\cdot\rangle^{r_1}(1-\chi')\chi\mathcal{F}(\varphi \hat{f}_0^*u)\rangle|\\
=\sup_{\theta\in\SSS(\RR^m),\, \|\theta\|_{\mathcal{F}^{-1}L^q(\RR^m)}\leq1}|\langle \varphi\langle D\rangle^{r_1}((1-\check{\chi}')\check{\chi})(D)\mathcal{F}\theta,\hat{f}_0^*u\rangle|.
\end{multline}
Since $(1-\check{\chi}')\check{\chi}$ has compact support, the Fourier multiplier $\langle D\rangle^{r_1}((1-\check{\chi}')\check{\chi})(D)$ considered as a $\Psi$DO has symbol in $S^{-\infty}(\RR^{2m})$ and consequently $\{\langle D\rangle^{r_1}((1-\check{\chi}')\check{\chi})(D)\mathcal{F}\theta\,|\, \theta\in\SSS(\RR^m),\, \|\mathcal{F}\theta\|_{L^q(\RR^m)}\leq 1\}$ is a bounded subset of $\DD_{L^q}(\RR^m)$, which, in turn, implies that $\{\varphi\langle D\rangle^{r_1}((1-\check{\chi}')\check{\chi})(D)\mathcal{F}\theta\,|\, \theta\in\SSS(\RR^m),\, \|\mathcal{F}\theta\|_{L^q(\RR^m)}\leq 1\}$ is a bounded subset of $\DD(\widetilde{O}_0)$. We infer that \eqref{equ-for-semofpulbkhsfs} is a continuous seminorm on $\DD'(\widetilde{O}_0)$ of $\hat{f}^*_0u$. Hence, we need to bound $\mathfrak{p}(\hat{f}^*_0u)$ with $\mathfrak{p}$ a continuous seminorm on $\DD'(\widetilde{O}_0)$. Incidentally, besides completing the estimates for $\mathfrak{p}_{r_1,p;\varphi,\chi}(\hat{f}^*_0u)$, this will also complete the proof of $(*)$ when $k=m$ since the seminorms of the form $\mathfrak{p}(\hat{f}^*_0u)$, $\mathfrak{p}$ continuous on $\DD'(\widetilde{O}_0)$, are the only remaining seminorms we need to estimate. Without loss in generality, we can assume that $\mathfrak{p}(\hat{f}^*_0u)=\sup_{\chi\in B}|\langle \hat{f}^*_0u,\chi\rangle|$ where $B$ is a bounded subset of $\DD_K(\RR^m)$ for some compact set $K\subseteq \widetilde{O}_0$. Since $L\cap \mathcal{N}_{\hat{f}_0}=\emptyset$ and in view of \eqref{nor-for-map-injec111111}, we can argue as in the proof of \cite[Theorem 3.21, CASE 3]{PP} to find an open set $\widetilde{U}_K\subseteq \widetilde{U}_0$ satisfying $\hat{f}_0(K)\subseteq \widetilde{U}_K$ and $C_K>1$ such that $\widetilde{V}_K:=\{(\eta',\eta'')\in\RR^n\,|\, |\eta''|>C_K|\eta'|\}$ satisfies $(\widetilde{U}_K\times \overline{\widetilde{V}_K})\cap \widetilde{L}=\emptyset$. Pick $\phi\in\DD(\widetilde{U}_K)$ satisfying $\phi=1$ on a neighbourhood of $\hat{f}_0(K)$. Choose $\psi\in \XX_0(\RR^n)$ such that $0\leq \psi\leq1$, $\supp\psi\subseteq \widetilde{V}_K\backslash B(0_n,1/2)$ and $\psi=1$ on $\{(\eta',\eta'')\in\RR^n\,|\, |\eta''|>2C_K|\eta'|\}\backslash B(0_n,1)$. Write
\begin{align*}
\sup_{\chi\in B}|\langle \hat{f}^*_0u,\chi\rangle|&=\sup_{\chi\in B}|\langle \hat{f}^*_0(\phi u),\chi\rangle|=\sup_{\chi\in B} \frac{1}{(2\pi)^n}\left|\iint_{\RR^n\times\RR^m}\mathcal{F}(\phi u)(\eta)e^{ix\eta'}\chi(x)d\eta dx\right|\\
&\leq I''_1+I''_2,
\end{align*}
where
\begin{align*}
I''_1&:=\sup_{\chi\in B} \frac{1}{(2\pi)^{n-m}}\left|\int_{\RR^n}\mathcal{F}(\phi u)(\eta)\psi(\eta)\mathcal{F}^{-1}\chi(\eta')d\eta\right|,\\
I''_2&:=\sup_{\chi\in B} \frac{1}{(2\pi)^{n-m}}\left|\int_{\RR^n}\mathcal{F}(\phi u)(\eta)(1-\psi(\eta))\mathcal{F}^{-1}\chi(\eta')d\eta\right|.
\end{align*}
Pick $\lambda>(n-m)/p$ such that $r_2\geq \lambda$. As in \eqref{ine-for-iprjklskr}, for $I''_1$ we infer
\begin{equation*}
I''_1\leq \|\langle\cdot\rangle^{r_2}\psi\mathcal{F}(\phi u)\|_{\mathcal{F}L^p(\RR^n)}\sup_{\chi\in B}\|\langle D_y\rangle^{-r_2}\langle D_{y''}\rangle^{\lambda}(\chi\otimes \langle D_{y''}\rangle^{-\lambda}\delta_{n-m})\|_{L^q(\RR^n)}.
\end{equation*}
In view of the Claim, $\langle D_{y''}\rangle^{-\lambda}\delta_{n-m}\in L^q(\RR^{n-m})$. The symbol of the Fourier multiplier $\langle D_y\rangle^{-r_2}\langle D_{y''}\rangle^{\lambda}$ is $\langle \eta\rangle^{-r_2}\langle \eta''\rangle^{\lambda}$ and it is straightforward to check that it satisfies the assumptions of the Lizorkin-Marcinkiewicz multiplier theorem \cite[Corollary 6.2.5, p. 444]{Grafakos}. Consequently, it is continuous on $L^q(\RR^n)$ which implies $I''_1\leq C'_4\mathfrak{p}_{r_2,p;\phi,\psi}(u)$ and the latter is a well-defined continuous seminorm on $\DD'^{r_2,p}_{\widetilde{L}}(\widetilde{U}_0)$ in view of the properties of $\psi$. We claim that $I''_2$ is a continuous seminorm on $\DD'(\widetilde{U}_0)$ of $u$. For this, it suffices to prove that
\begin{align*}
\DD'(\widetilde{U}_0)\rightarrow[0,\infty),\quad v\mapsto\sup_{\chi\in B}|\langle v,\phi(\operatorname{Id}-\check{\psi}(D_y))\langle D_{y''}\rangle^{\lambda}(\chi\otimes\langle D_{y''}\rangle^{-\lambda}\delta_{n-m})\rangle|,
\end{align*}
is a well-defined continuous seminorm on $\DD'(\widetilde{U}_0)$. We are going to show that $\{(\operatorname{Id}-\check{\psi}(D_y))\langle D_{y''}\rangle^{\lambda}(\chi\otimes\langle D_{y''}\rangle^{-\lambda}\delta_{n-m})\,|\, \chi\in B\}$ is a bounded subset of $\DD_{L^q}(\RR^n)$ which, as before, proves the desired claim. For $\beta\in\NN^n$, we infer
\begin{multline*}
\partial^{\beta}_y(\operatorname{Id}-\check{\psi}(D_y))\langle D_{y''}\rangle^{\lambda}(\chi\otimes\langle D_{y''}\rangle^{-\lambda}\delta_{n-m})\\
=\partial^{\beta}_y(\operatorname{Id}-\check{\psi}(D_y))\langle D_{y'}\rangle^{-|\beta|-\lambda}\langle D_{y''}\rangle^{\lambda}(\langle D_{y'}\rangle^{|\beta|+\lambda}\chi\otimes\langle D_{y''}\rangle^{-\lambda}\delta_{n-m}).
\end{multline*}
The symbol of the Fourier multiplier is $i^{|\beta|}\eta^{\beta}(1-\check{\psi}(\eta))\langle \eta'\rangle^{-|\beta|-\lambda}\langle\eta''\rangle^{\lambda}$ and it straightforward to verify that it satisfies the assumptions of the Lizorkin-Marcinkiewicz multiplier theorem \cite[Corollary 6.2.5, p. 444]{Grafakos}. Consequently, it is continuous on $L^q(\RR^n)$, which immediately implies the desired result. This completes the proof of the case $k=m$.\\
\indent \underline{The case when $k\in\ZZ_+$.} Write $\hat{f}_0=\hat{f}_2\circ\hat{f}_1$ where $\hat{f}_1:\widetilde{O}_0\rightarrow\hat{f}_1(\widetilde{O}_0)\subseteq \RR^k$, $\hat{f}_1(x',x'')=x'$, is surjective and $\hat{f}_2:\hat{f}_1(\widetilde{O}_0)\rightarrow \widetilde{U}_0$, $\hat{f}_2(x')=(x',0_{n-k})$, is injective. Notice that $\hat{f}^*_0=\hat{f}^*_1\hat{f}^*_2$ both on the spaces of smooth functions as well as on the spaces of covectors. In view of the second case $\hat{f}^*_2:\DD'^{r_2,p}_{\widetilde{L}}(\widetilde{U}_0)\rightarrow\DD'^{r_1,p}_{\hat{f}^*_2\widetilde{L}}(\hat{f}_1(\widetilde{O}_0))$ is well-defined and continuous since $\mathcal{N}_{\hat{f}_2}\cap \widetilde{L}=\mathcal{N}_{\hat{f}_0}\cap \widetilde{L}=\emptyset$ and $r_2-r_1>(n-k)/p$, $r_2>(n-k)/p$. In view of the first case $\hat{f}^*_1:\DD'^{r_1,p}_{\hat{f}^*_2\widetilde{L}}(\hat{f}_1(\widetilde{O}_0))\rightarrow \DD'^{r_1,p}_{\hat{f}^*_0\widetilde{L}}(\widetilde{O}_0)$ is well-defined and continuous since $\mathcal{N}_{\hat{f}_1}\cap \hat{f}^*_2\widetilde{L}=(\hat{f}_1(\widetilde{O}_0)\times\{0_k\})\cap \hat{f}^*_2\widetilde{L}=\emptyset$. This completes the proof.\\
\\
\noindent \textbf{Proof of Claim.} In the proof of \cite[Proposition A.1]{PP} (see the identity \cite[(A.2)]{PP}) it is shown that for any $\lambda>0$, $\langle D\rangle^{-\lambda}\delta_d\in L^1(\RR^d)$ and it is given by
\begin{equation*}
(\langle D\rangle^{-\lambda}\delta_d)(x)=\frac{1}{\Gamma(\lambda/2)2^d\pi^{d/2}}\int_0^{\infty}t^{(\lambda-d-2)/2}e^{-t} e^{-|x|^2/(4t)} dt,\quad x\in\RR^d\backslash\{0\}
\end{equation*}
(cf. \cite[Proposition 2, p. 132]{stein}). When $\lambda>d/p$, the Minkowski integral inequality implies
\begin{align*}
\|\langle D\rangle^{-\lambda}\delta_d\|_{L^q(\RR^d)}&\leq \frac{1}{\Gamma(\lambda/2)2^d\pi^{d/2}}\int_0^{\infty}t^{(\lambda-d-2)/2}e^{-t} \left(\int_{\RR^d}e^{-q|x|^2/(4t)}dx\right)^{1/q} dt\\
&=\frac{2^{d/q}\|e^{-|\cdot|^2}\|_{L^1(\RR^d)}^{1/q}}{\Gamma(\lambda/2)2^d\pi^{d/2}q^{d/(2q)}}\int_0^{\infty}t^{(q\lambda-qd-2q+d)/(2q)}e^{-t} dt<\infty,
\end{align*}
since $(q\lambda-qd-2q+d)/(2q)>-1$ and the proof of the claim is complete.
\end{proof}

The typical application of the result is for extending elements in $\DD'^{r,p}_L(U)$ to open sets in higher dimension which corresponds to the pullback by the projection of the higher dimensional open set onto $U$. The other application is to restricting elements in $\DD'^{r,p}_L(U)$ to lower dimensional hyperplanes which corresponds to the pullback by the imbedding of the hyperplane in $U$; notice that in this case the theorem states that there is a loss in $L^p$-Sobolev regularity of $(n-\operatorname{dim}(\mbox{plane}))/p$ and can be done when $r>(n-\operatorname{dim}(\mbox{plane}))/p$.


\begin{thebibliography}{999}

\bibitem{adams} R. A. Adams, J. J. F. Fournier, \textit{Sobolev spaces}, Academic press, 2003.

\bibitem{D1} C. Brouder, N. V. Dang, F. H\'elein, \textit{Continuity of the fundamental operations on distributions having a specified wave front set (with a counterexample by Semyon Alesker)}, Stud. Math. 232(3) (2016), 201-226.

\bibitem{DB} Y. Dabrowski, C. Brouder, \textit{Functional properties of H\"omander’s space of distributions having a specified wavefront set}, Commun. Math. Phys. 332 (2014), 1345-1380.

\bibitem{dap-r-scl} C. Dappiaggi, P. Rinaldi, F. Sclavi, \textit{Besov wavefront set}, Anal. Math. Phys. 13(6) (2023), Paper No. 95.

\bibitem{DPV} P. Dimovski, S. Pilipovi\'c, J. Vindas, \textit{New distribution spaces associated to translation-invariant Banach spaces},
Monatsh. Math. 177(4) (2015), 495–515.

\bibitem{DPP} P. Dimovski, S. Pilipovi\'c, B. Prangoski, \textit{On a class of Mikhlin multipliers which do not preserve $L^1$-, $L^{\infty}$-regularity and continuity}, J. Math. Anal. Appl. 553(1) (2026), Article ID 129971.

\bibitem{DPPV-TMIB} P. Dimovski, S. Pilipovi\'c, B. Prangoski, J. Vindas, \textit{Translation-modulation invariant Banach spaces of ultradistributions}, J. Fourier Anal. Appl. 25(3) (2019), 819–841.

\bibitem{fefferman} C. Fefferman, \textit{$L^p$ bounds for pseudo-differential operators}, Isr. J. Math. 14 (1973), 413-417.

\bibitem{Grafakos} L. Grafakos, \emph{Classical Fourier analysis}, third edition, Springer, New York, 2014.

\bibitem{hormander} L. H\"ormander, \textit{Fourier integral operators. I}, Acta Math. 127 (1971), 79-183.

\bibitem{hor} L. H\"ormander, \textit{The analysis of linear partial differential operators I. Distribution theory and fourier analysis}, Springer, 2003.

\bibitem{hor2} L. H\"ormander, \textit{The analysis of linear partial differential operators III. Pseudo-differential operators}, Springer, 2007.

\bibitem{lee} J. M. Lee, \textit{Introduction to Smooth Manifolds}, Springer, New York, 2013.

\bibitem{PP} S. Pilipovi\'c, B. Prangoski, \textit{Spaces of distributions with Sobolev wave front in a fixed conic set: compactness, pullback by smooth maps and the compensated compactness theorem}, preprint, arXiv:2408.10741

\bibitem{triebel} H. Triebel, \textit{Theory of function spaces}, Springer, Basel, 2010

\bibitem{sawano} Y. Sawano, \textit{Theory of Besov spaces}, Springer, Singapore, 2018.

\bibitem{stein} E. M. Stein, \textit{Singular integrals and differentiability properties of functions}, Princeton University Press, Princeton, New Jersey, 1970.

\bibitem{Tutic} S. Tuti\'c, \textit{On the continuity of the product of distributions in local Sobolev spaces}, J. Pseudo-Differ. Oper. Appl. 17(2) (2026), Paper No. 30.
    
\end{thebibliography}
\end{document}